\documentclass{article}   	
\usepackage[margin=1in]{geometry}                	
\usepackage[affil-it]{authblk}     
\usepackage{graphicx}
\usepackage{caption, subcaption} 
\usepackage{amssymb, amsmath, amsthm}
\usepackage{color, todonotes}
\usepackage{hyperref}

\newtheorem{theorem}{Theorem}
\newtheorem{conjecture}{Conjecture}

\graphicspath{{./figures/}}

\newcommand{\bx}{\mathbf{x}}
\newcommand{\by}{\mathbf{y}}
\newcommand{\bz}{\mathbf{z}}
\newcommand{\bfor}{\mathbf{f}}
\newcommand{\bth}{\boldsymbol{\theta}}

\title{Improving the Convergence of the Iterative Ensemble Kalman Filter by Resampling}

\author[1]{Jiacheng Wu}
\author[1,2,3]{Jian-Xun Wang}
\author[1]{Shawn C. Shadden}
\affil[1]{Department of Mechanical Engineering, University of California, Berkeley}
\affil[2]{Department of Aerospace and Mechanical Engineering, University of Notre Dame}
\affil[3]{Center of Informatics and Computational Science, University of Notre Dame}

\date{}

\begin{document}
\maketitle

\section*{Abstract}
The iterative ensemble Kalman filter (IEnKF) is widely used in inverse problems to estimate system parameters from limited observations. 
However, the IEnKF, when applied to nonlinear systems, can be plagued by poor convergence. 
Here we provide a comprehensive convergence analysis of the IEnKF and propose a new method to improve its convergence. 
A theoretical analysis of the standard IEnKF is presented and we demonstrate that the interaction between the nonlinearity of the forward model and the diminishing effect of the Kalman updates results in ``early stopping" of the IEnKF, i.e. the Kalman operator converges to zero before the innovation is minimized. 
The steady state behavior of the early stopping phenomenon and its relation to observation uncertainty is demonstrate. 
We then propose an approach to prevent the early stopping by perturbing the covariance with hidden parameter ensemble resampling. 
The ensemble mean and covariance are kept unchanged during the resampling process, which ensures the Kalman operator at each iteration maintains a correct update direction. 
We briefly discuss the influence higher moments, such as kurtosis, of the resampling distribution on algorithm convergence. 
Parallel to the above developments, an example problem is presented to demonstrate the early stopping effect, and the application and merit of the proposed resampling scheme.

\section{Introduction}


The Kalman filter provides optimal estimation for linear dynamics with Gaussian noise~\cite{kalman1960new} and has been widely used in engineering applications.  
There have been several variants of this classic method aimed to improve its generality and efficiency. 
The extended Kalman filter (EKF)~\cite{law2015data} and unscented Kalman filter (UKF)~\cite{julier2004unscented} were introduced to better address state estimation for nonlinear systems, and the ensemble Kalman filter (EnKF)~\cite{evensen1994sequential, bergemann2010localization, bergemann2010mollified} was proposed to reduce computational cost by utilizing ensemble-based covariances. 
The iterative ensemble Kalman filter (IEnKF)~\cite{iglesias2013ensemble} was developed to handle nonlinear inverse problems and leverage the computational efficiency of ensemble-based sampling. 
The IEnKF has since been applied to highly nonlinear inverse problems in areas such as turbulence~\cite{kato2015data, wang2016data}, geophysics~\cite{shi2015parameter} and biomedical engineering~\cite{wang2018data}. 
Performance analyses of the IEnKF can be found in \cite{schillings2017convergence, schillings2017analysis, majda2018performance}.

Despite its broad applicability, the IEnKF can suffer from poor convergence and stability. A major reason the IEnKF can fail to provide accurate estimation is due to a progressively diminished estimation (shrinking) of the covariance. This was initially addressed by including covariance ``inflation'' in the updates \cite{anderson1999monte}; however, tuning the inflation parameter can be inefficient. More recently, several adaptive covariance inflation methods have been proposed, which tune the inflation based on the innovation \cite{anderson2007adaptive, anderson2009spatially, berry2013adaptive, zhen2015adaptive, ying2015adaptive}. Similarly, \cite{tong2015nonlinear} proposed an adaptive covariance inflation method where the inflation parameter depends on both the innovation and the covariance between the observed and unobserved components. 

The objective of the work herein is to provide a mathematical explanation of the covariance shrinking effect observed with the IEnKF, and based on that, propose a new method to improve its convergence without covariance inflation. To achieve this goal, in \S\ref{sec:need_for_iteration} we first motivate the need for iterations of the EnKF in nonlinear inverse problems from a constrained optimization viewpoint. Then in \S\ref{sec:early-stopping}, we explain how the interplay between covariance shrinkage and the effect of the nonlinearity of the forward model can lead to ``early stopping'' whereby the method fails to converge before the innovation is minimized. In section \ref{sec:resampling-method}, we show that this early stopping can be prevented and convergence can be improved by adding an ensemble resampling step where first and second moments of the hidden parameter ensemble are kept unchanged. In \S\ref{sec:kurtosis}, different resampling distributions are compared to demonstrate the influence of higher order moments (particularly kurtosis) on the convergence of IEnKF. 

\section{The IEnKF Method} \label{sec:need_for_iteration} 

Consider a system described by a known {\em forward model}
\begin{equation}
\bx = \bfor(\boldsymbol{\bth}) \label{eq:F} 
\end{equation}
where $\bx \in \mathbb{R}^{d_{\bx}}$ is the {\em system state} and $\bth \in \mathbb{R}^{d_{\bth}}$ are the {\em model parameters}. 
Assume we have knowledge of the system state through {\em observations}
\begin{equation}
\by = H\bx + \boldsymbol{\epsilon} \label{eq:H} 
\end{equation}
where $\by \in \mathbb{R}^{d_y}$ and $\boldsymbol{\epsilon} \in \mathbb{R}^{d_y}$ represents measurement error. 
Without loss of generality, we assume the {\em observation operator} $H$ is linear. 
If $H$ is nonlinear, the nonlinearity can be absorbed into the nonlinearity of the forward model by redefining the state variable.

The {\em inverse problem} seeks to estimate $\bth$ from observations $\by$. 
This estimation can be performed using the IEnKF, which employs a {\em two-stage} iterative estimation process. 
The first stage entails an ensemble-based Kalman {\em update}, and the second stage entails a subsequent {\em prediction} to ensure an overall update that is consistent with the forward model. 
This two-stage process is iteratively repeated until convergence, as reviewed below. 
 
\subsection{Update stage}

Assume we have a set of {\em prior ensembles} for $\bth$ and $\bx$: 
$$
\left\{ \hat{\bth}_{t}^{(j)} \right\}_{j=1}^J \quad \textup{and} \quad \left\{ \hat{\bx}_{t}^{(j)} \right\}_{j=1}^J \;.
$$
Index $(j)$ denotes an ensemble member, $J$ is the number of ensembles, and index $t$ denotes the {\em iteration} number in what will be an iterative update process, and the hat $\hat{(\cdot)}$ denotes a {\em prior} estimate.
The difference between the observations $\by_t^{(j)}$ and reconstructed output $H\hat{\bx}_t^{(j)}$ is often termed the {\em innovation}, which is used to define an ensemble-based Kalman {\em update} to produce {\em posterior ensembles} for $\bth$ and $\bx$: 
\begin{subequations}
\begin{align}
\bth_t^{(j)} & =\hat{\bth}_{t}^{(j)} + C_{\hat{\bth}_t \hat{\bx}_t} H^T \left( HC_{\hat{\bx}_t \hat{\bx}_t} H^T + \Gamma \right)^{-1}\left( \by_t^{(j)} - H\hat{\bx}_t^{(j)} \right) \;, \label{eq:update-theta}\\
\bx_t^{(j)} & =\hat{\bx}_{t}^{(j)} + C_{\hat{\bx}_t \hat{\bx}_t} H^T \left( HC_{\hat{\bx}_t \hat{\bx}_t} H^T + \Gamma \right)^{-1}\left( \by_t^{(j)} - H\hat{\bx}_t^{(j)} \right) \;, \label{eq:update-x}
\end{align}
\label{eq:KF}
\end{subequations}
where $C_{\hat{\bth}_t \hat{\bx}_t}$, $C_{\hat{\bx}_t \hat{\bx}_t}$ are the discrete covariance matrices derived from the prior ensembles:
\begin{align}
C_{\hat{\bth}_t \hat{\bx}_t}  &= \frac{1}{J}\sum_{j=1}^J \left( \hat{\bth}_{t}^{(j)} - \bar{\hat{\bth}}_{t} \right) \left(  \hat{\bx}_{t}^{(j)} - \bar{\hat{\bx}}_{t}  \right)^T \;, \nonumber \\
C_{\hat{\bx}_t \hat{\bx}_t}  &= \frac{1}{J}\sum_{j=1}^J \left( \hat{\bx}_{t}^{(j)} - \bar{\hat{\bx}}_{t} \right) \left(  \hat{\bx}_{t}^{(j)} - \bar{\hat{\bx}}_{t}  \right)^T\;.
\end{align}
The {\em bar} notation $\bar{(\cdot)}$ denotes {\em ensemble mean}. 
For each ensemble member $j$, the observation $\by_t^{(j)}$ is drawn from a normal distribution $\mathcal{N}(\bar{\by}, \Gamma)$, where $\Gamma$ denotes the covariance of the observation error $\boldsymbol{\epsilon}$ (see \cite{iglesias2013ensemble} for details).  

\subsection{Prediction stage}

The Kalman update \eqref{eq:update-x} may generate posterior estimates $\bx_t^{(j)}$ that do not satisfy the forward model, Eq.~\eqref{eq:F}. Therefore, instead of using Eq.~\eqref{eq:update-x}, a {\em prediction} is used to apply the forward model to the current posterior parameter estimates to generate prior states for the next iteration step. Namely, the prior ensembles for the next step $\left\{ \hat{\bth}_{t+1}^{(j)} \right\}_{j=1}^J$ and $\left\{ \hat{\bx}_{t+1}^{(j)} \right\}_{j=1}^J$ are obtained by setting
\begin{subequations}
\begin{align}
\hat{\bth}_{t+1}^{(j)} &= \bth_{t}^{(j)}\;, \\
\hat{\bx}_{t+1}^{(j)} &= \bfor(\bth_{t}^{(j)})\;.
\end{align}
\label{eq:forward}
\end{subequations}

\subsection{Iterative process}
The results of Eq.~\eqref{eq:forward} are plugged back into Eq.~\eqref{eq:KF}, with $t \leftarrow t+1$, and the process is repeated until some stopping criterion is satisfied, e.g., the innovation becomes less than some user-defined error tolerance ($tol$): 
\begin{equation}
\left\|\bar{\by} - H \bx_t \right\|^2 < tol ~~  \textup{with} ~~ \bx_t \doteq \bfor(\bar{\bth}_t) \;,
\end{equation}
where $\bar{\by}$ is the mean of the observation ensemble $\left\{ \by_t^{(j)} \right\}_{j=1}^J$, which is independent of $t$. 
Here, the overall estimation of the unknown parameters at step $t$ is computed as the posterior ensemble mean 
\begin{equation}
\bar{\bth}_t = \frac{1}{J} \sum_{j=1}^J \bth_t^{(j)}\;,
\end{equation}
Note that when the forward model is linear (i.e., $\bfor(\bth) = F\bth$), the prediction step (\ref{eq:forward}) is not necessary, because 
\begin{align}
\bx_t^{(j)} & =\hat{\bx}_{t}^{(j)} + C_{\hat{\bx}_t \hat{\bx}_t} H^T \left( HC_{\hat{\bx}_t \hat{\bx}_t} H^T + \Gamma \right)^{-1}\left( \by_t^{(j)} - H\hat{\bx}_t^{(j)} \right) \nonumber \\
& = F\hat{\bth}_{t}^{(j)} + FC_{\hat{\bth}_t \hat{\bx}_t} H^T \left( HC_{\hat{\bx}_t \hat{\bx}_t} H^T + \Gamma \right)^{-1}\left( \by_t^{(j)} - H\hat{\bx}_t^{(j)} \right) \nonumber \\
& = F\bth_{t}^{(j)} \;.
\end{align}
However, when $\bfor(\cdot)$ is nonlinear, the posterior estimates from Eq.~\eqref{eq:KF} will not satisfy Eq.~\eqref{eq:F} in general, which is why an iterative estimation process is needed. It is the convergence of this iterative process for nonlinear problems that is the focus of this paper.  

\subsection{Example} \label{sec:example} 

Here we introduce an example that demonstrates the ``early stopping'' of the IEnKF for nonlinear systems. 
Consider the following forward model:  
\begin{equation}
\begin{bmatrix}
    x_1   \\
    x_2
\end{bmatrix}
=
\bfor(\bth)
=
\begin{bmatrix}
     \exp(-(\theta_1 + 1)^2 - (\theta_2 + 1)^2)   \\
     \exp(-(\theta_1 - 1)^2 - (\theta_2 - 1)^2)
\end{bmatrix}\;,
\label{eq:F_example}
\end{equation}
with observation operator
\begin{equation} \label{eq:H_example}
H = [-1.5, -1.0] \;.
\end{equation}
Assume the observation data has mean $\bar{y} = -1$ and uncertainty $\Gamma = 0.01$. 

Figure \ref{fig:fig_theta_convergence_no_resampling} and \ref{fig:fig_x_convergence_no_resampling} show, respectively, the solution path for $\bth$ in parameter space, and the solution path for $\bx$ in state space, during progressive iterations of the IEnKF. 
It can be observed in Fig.~\ref{fig:fig_theta_convergence_no_resampling} that the parameter estimate does not converge to a solution where $\left\|\bar{\by} - H\bfor(\bth_t) \right\|^2 = 0$. 
In Figure \ref{fig:fig_x_convergence_no_resampling} the prior mean and posterior mean are sequentially plotted at each iteration. It can be observed that the solution oscillates between these two means. 
In particular, the update stage makes the solution approach $\left\|\bar{\by} - H\bx_t \right\|^2 = 0$ shown in red, and the prediction stage makes the solution approach the dotted region where the forward model Eq.~\eqref{eq:F} is satisfied. 
{\em This example demonstrates that although innovation is not minimized, further iterations of the IEnKF will not improve the estimation of $\bth$, nor enable the state estimate converge to a value that simultaneously satisfies the forward model and minimizes misfit with the observations. } 
These two phenomena will be considered more rigorously in the next section.

\begin{figure}
\begin{subfigure}[t]{0.45\textwidth}
\includegraphics[width=\textwidth]{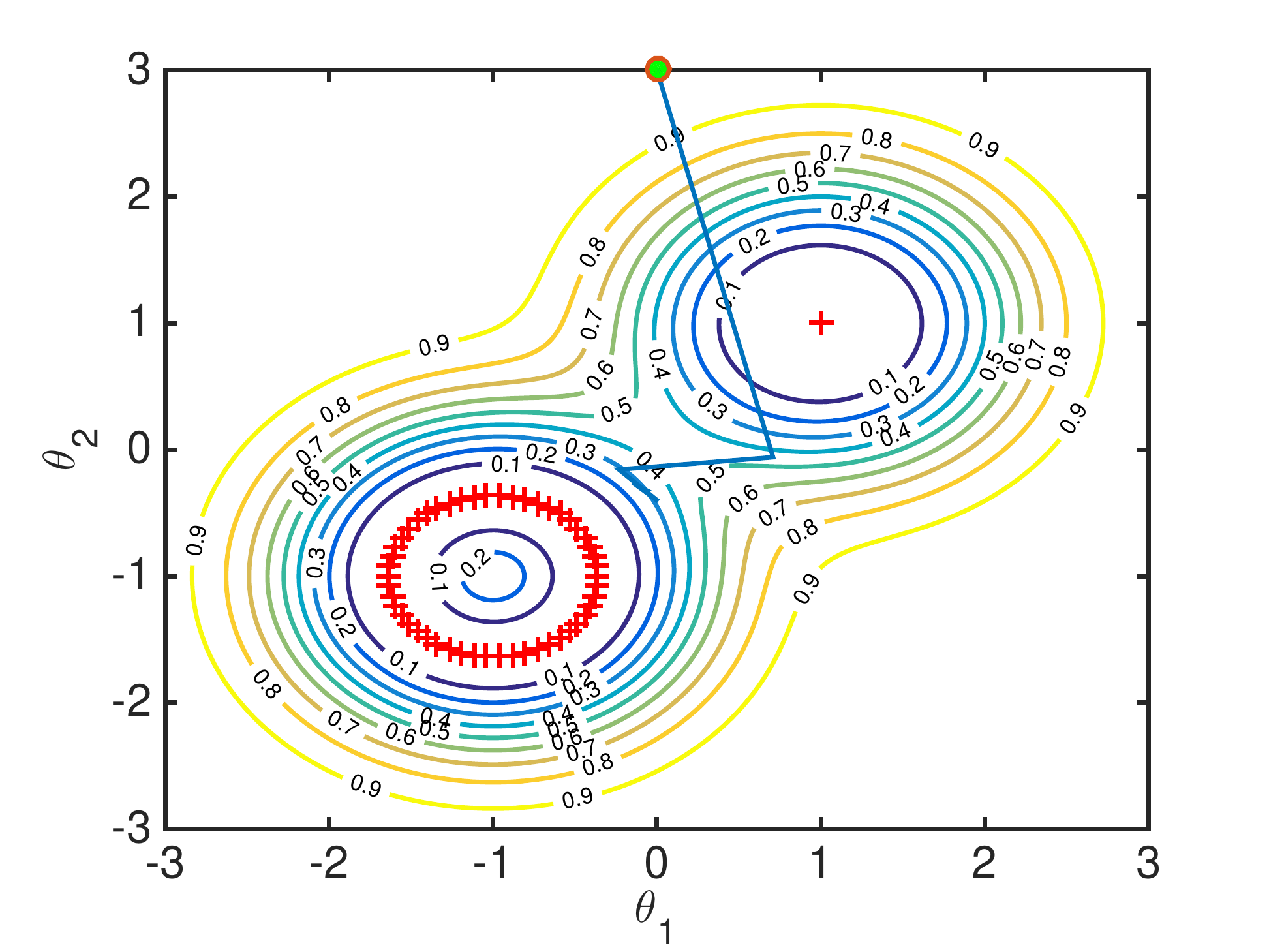}
\caption{The solution path of the ensemble mean $\bar{\bth}_t$ at each iteration $t$ plotted along with the level sets of $\left\|\bar{\by} - H\bfor(\bth) \right\|^2$. The green circle denotes the starting value. The red ``+" highlight optimal locations where $\left\|\bar{\by} - H\bfor(\bth) \right\|^2 = 0$. }
\label{fig:fig_theta_convergence_no_resampling}
\end{subfigure}
\hfill
\begin{subfigure}[t]{0.45\textwidth}
\includegraphics[width=\textwidth]{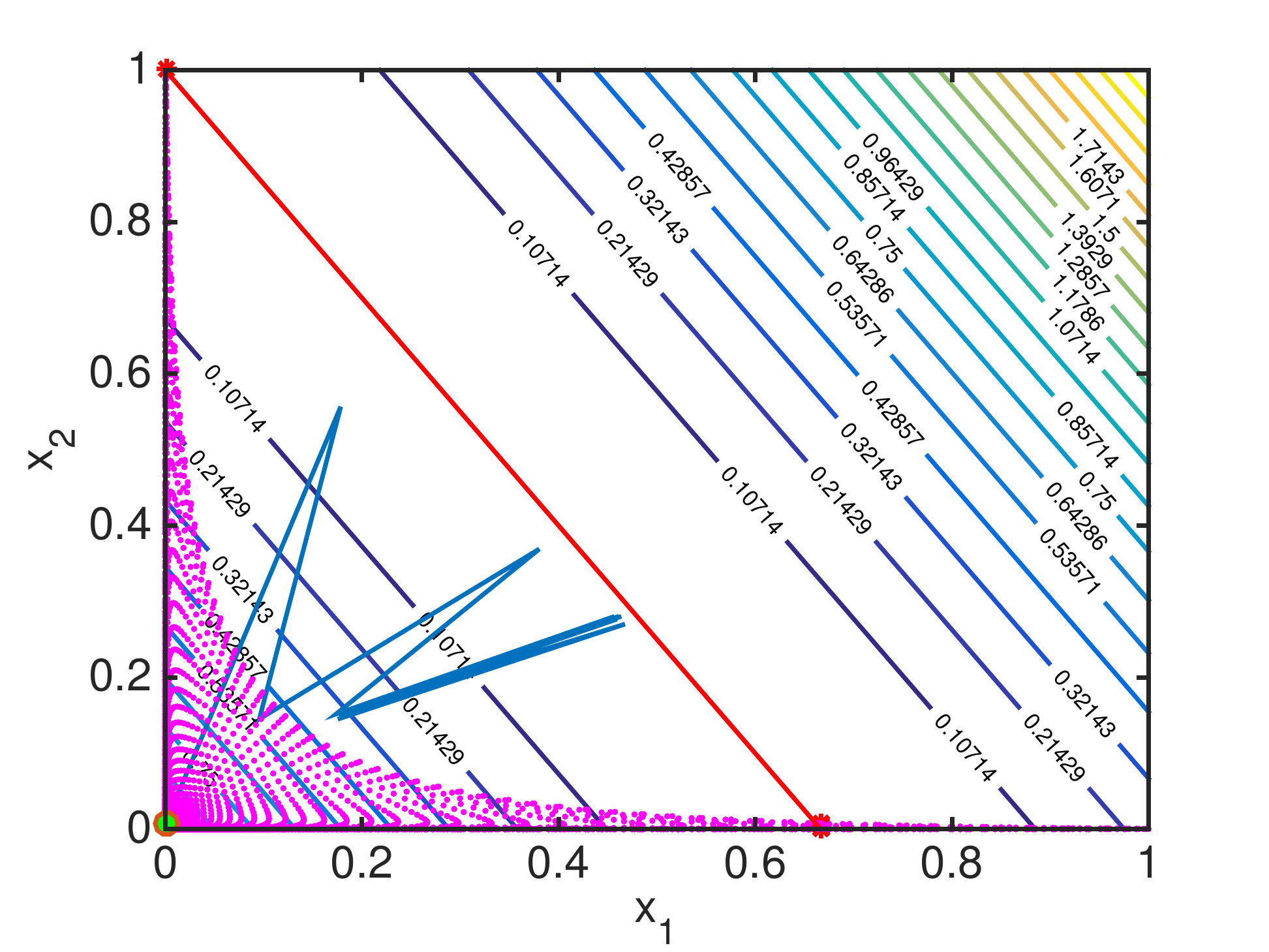}
\subcaption{The solution path of the prior mean $\bar{\hat{\bx}}_t$ and posterior mean $\bar{\bx}_t$ at each iteration $t$, along with level sets of $\left\|\bar{\by} - H\bx \right\|^2$. The green circle denotes the starting value. The red line highlights the optimal locations where $\bar{\by} - H\bx = 0$. The dotted region represents the image of the forward mapping $\bfor(\cdot)$. }
\label{fig:fig_x_convergence_no_resampling}
\end{subfigure}
\caption{Convergence of the IEnKF applied to the model problem demonstrating (a) parameter and (b) state values do not converge to the desired solution (red locations).} \label{fig:convergence_no_resampling}
\label{fig:convergence_no_resampling}
\end{figure}

\section{Early Stopping of the IEnKF} \label{sec:early-stopping}
In this section, we describe why the standard IEnKF leads to poor convergence when applied to nonlinear inverse problems. 
We first demonstrate how sequential Kalman updating alone affects the covariance matrices and resulting Kalman operator. 
We then demonstrate how the full IEnKF procedure, involving both Kalman update and prediction, affects the covariance matrices and Kalman operator. 

\subsection{Shrinking covariance}\label{sec:standard_iteration}
If we ignore the prediction stage Eq.~\eqref{eq:forward} and only consider the Kalman update Eq.~\eqref{eq:KF}, it can be shown that the covariance matrix $C_{H\hat{\bx} ,\hat{\bth}}$ resulting from two sequential Kalman updates is
\begin{equation}
C_{H\hat{\bx}_{t+1}, \hat{\bth}_{t+1}} = \left(  I - HC_{\hat{\bx}_t, \hat{\bx}_t} H ^T \left( HC_{\hat{\bx}_t,\hat{\bx}_t} H^T+ \Gamma \right)^{-1} \right) C_{H\hat{\bx}_t, \hat{\bth}_t}\;,
\end{equation}
where $I$ denotes the identity matrix. 
The above equation can be written as
\begin{equation}
C_{H\hat{\bx}_{t+1}, \hat{\bth}_{t+1}} =P_t C_{H\hat{\bx}_t, \hat{\bth}_t}\;,
\end{equation}
where
\begin{align}
P_t &\doteq  I - HC_{\hat{\bx}_t, \hat{\bx}_t} H ^T \left( HC_{\hat{\bx}_t,\hat{\bx}_t} H^T+ \Gamma \right)^{-1} \nonumber \\
&= \Gamma \left( HC_{\hat{\bx}_t,\hat{\bx}_t} H^T+ \Gamma \right)^{-1}. 
\end{align}
Note that the observation uncertainty usually takes the form $\Gamma = \alpha I$ with $\alpha > 0$. Thus, an upper bound for the Frobenius norm $\|\cdot\|_F$ of $C_{H\hat{\bx}_{t+1}, \hat{\bth}_{t+1}}$ can be derived 
\begin{align}
\left\|  C_{H\hat{\bx}_{t+1}, \hat{\bth}_{t+1}} \right\|_F & = \left\| P_t  C_{H\hat{\bx}_{t}, \hat{\bth}_{t}} \right\|_F  \nonumber \\
& \leq \frac{\alpha}{\lambda_{min} \left( HC_{\hat{\bx}_t,\hat{\bx}_t} H^T \right)+ \alpha} \left\| C_{H\hat{\bx}_{t}, \hat{\bth}_{t}} \right\|_F \nonumber \\
& \leq \prod_{\tau=1}^t \frac{\alpha}{\lambda_{min} \left( HC_{\hat{\bx}_\tau,\hat{\bx}_\tau} H^T \right)+ \alpha} \left\| C_{H\hat{\bx}_{1}, \hat{\bth}_{1}} \right\|_F\;.
\end{align}
If there exists an uniform lower bound $\delta > 0$ such that 
\begin{equation}
\lambda_{min} \left( HC_{\hat{\bx}_\tau,\hat{\bx}_\tau} H^T \right) > \delta,~~ \forall \tau = 1,2,...\;,
\label{ineq:conv_condition}
\end{equation}
then 
\begin{align}
\left\| C_{H\hat{\bx}_{t+1}, \hat{\bth}_{t+1}} \right\|_F 
\leq \left( \frac{\alpha}{\delta+ \alpha}\right)^t \left\|C_{H\hat{\bx}_{1}, \hat{\bth}_{1}} \right\|_F\;.
\end{align}
Therefore, $C_{H\hat{\bx}_{t}, \hat{\bth}_{t}}$ or $C_{\hat{\bth}_{t},H\hat{\bx}_{t}}$ will approach to zero as the iteration number increases, i.e.,
\begin{equation}
 \lim_{t \rightarrow +\infty} \left\| C_{H\hat{\bx}_{t}, \hat{\bth}_{t}} \right\|_F = 0\;.
\end{equation}
Similarly, it can be shown that $C_{\hat{\bx}_{t}, H\hat{\bx}_{t}}$  will also converge to a zero matrix as $t$ increases, unless Eq.~\eqref{ineq:conv_condition} is violated. 
By checking the Kalman update Eq.~\eqref{eq:KF}, it can be noted that the convergence of $C_{\hat{\bth}_{t}, H\hat{\bx}_{t}}$ and $C_{\hat{\bx}_{t}, H\hat{\bx}_{t}} $ to zero matrices will cause the Kalman gains 
\begin{align} 
K_t & = C_{\hat{\bth}_t \hat{\bx}_t} H^T \left( HC_{\hat{\bx}_t \hat{\bx}_t} H^T + \Gamma \right)^{-1}\;, \label{eq:kalman-gain-theta}\\
K'_t & =  C_{\hat{\bx}_t \hat{\bx}_t} H^T \left( HC_{\hat{\bx}_t \hat{\bx}_t} H^T + \Gamma \right)^{-1}\;, \label{eq:kalman-gain-x}
\end{align}
to approach to zero as $t$ increases. This implies that further iterations will not update the parameter estimation $\bth_t$, regardless of the innovation $\|\bar{\by}-Hx_t\|$. 

\subsection{Effect of forward model nonlinearity} \label{sec:nonlinearity}
We now consider the prediction step Eq.~\eqref{eq:forward} and show how the nonlinearity of the forward model affects the convergence of the IEnKF. 
First, consider the prior ensembles for the current step 
$$
\left\{ \hat{\bth}_{t}^{(j)} \right\}_{j=1}^J \quad \textup{and} \quad
\left\{ \hat{\bx}_{t}^{(j)} \right\}_{j=1}^J. 
$$ 
The mean of the prior ensembles can be computed as 
\begin{align}
\bar{\hat{\bth}}_t &= \frac{1}{J}\sum_{j=1}^J \hat{\bth}_t^{(j)}\;, \label{eq:mean_theta_t}\\
\bar{\hat{\bx}}_t & = \frac{1}{J}\sum_{j=1}^J \bfor\left( \hat{\bth}_t^{(j)} \right) \nonumber \\
& = \frac{1}{J}\sum_{j=1}^J \bfor\left( \bar{\hat{\bth}}_t \right) + \nabla \bfor\left( \bar{\hat{\bth}}_t \right) \left( \hat{\bth}_t^{(j)} - \bar{\hat{\bth}}_t \right) 
+ \mathcal{O}\left( \left\| \hat{\bth}_t^{(j)} - \bar{\hat{\bth}}_t  \right\|^2 \right) \nonumber \\
& = \bfor\left( \bar{\hat{\bth}}_t \right)  + \nabla \bfor\left( \bar{\hat{\bth}}_t \right) \frac{1}{J}\sum_{j=1}^J   \left( \hat{\bth}_t^{(j)} - \bar{\hat{\bth}}_t \right) 
+ \mathcal{O}\left( \left\| \hat{\bth}_t^{(j)} - \bar{\hat{\bth}}_t  \right\|^2 \right) \nonumber \\
& = \bfor\left( \bar{\hat{\bth}}_t \right) + \mathcal{O}\left( \left\| \hat{\bth}_t^{(j)} - \bar{\hat{\bth}}_t  \right\|^2 \right) \;, \label{eq:mean_x_t} 
\end{align}
by Taylor expansion. 
Recall that update for $\bth$ is given by
\begin{equation}
\hat{\bth}_{t+1}^{(j)} = \bth_t^{(j)} =\hat{\bth}_{t}^{(j)} + K_t\left( \by_t^{(j)} - H\hat{\bx}_t^{(j)} \right) \;. \label{eq:theta_update}
\end{equation}
The mean of the prior ensemble for $\bth$ at step $t+1$ can be computed as
\begin{align}
\bar{\hat{\bth}}_{t+1} & = \frac{1}{J}\sum_{j=1}^J \hat{\bth}_{t+1}^{(j)}  
 = \frac{1}{J}\sum_{j=1}^J {\bth}_{t}^{(j)} 
 = \bar{\hat{\bth}}_{t} + K_t \left( \bar{\by} - H\bar{\hat{\bx}}_t \right) \;,
\label{eq:mean_theta_t_1}
\end{align}
Combining (\ref{eq:mean_theta_t}), (\ref{eq:mean_x_t}), (\ref{eq:theta_update}), (\ref{eq:mean_theta_t_1}), we can obtain the deviation of each ensemble member to its ensemble mean at the next step $t+1$ as
\begin{align}
\hat{\bth}_{t+1}^{(j)}- \bar{\hat{\bth}}_{t+1}  &= \hat{\bth}_{t}^{(j)} - \bar{\hat{\bth}}_{t}   
 - K_t \left( H\hat{\bx}_t^{(j) }- H\bar{\hat{\bx}}_t  \right) + K_t \left( \by_t^{(j)} - \bar{\by}  \right) \;, \label{eq:theta_deviation}  \\
\hat{\bx}_{t+1}^{(j)} - \bar{\hat{\bx}}_{t+1} & = \bfor\left( \hat{\bth}_{t+1}^{(j)} \right) - \bfor\left( \bar{\hat{\bth}}_{t+1} \right) + \mathcal{O}\left( \left\| \hat{\bth}_{t+1}^{(j)} - \bar{\hat{\bth}}_{t+1}  \right\|^2 \right)  \nonumber \\
& = \nabla \bfor\left(\bar{\hat{\bth}}_{t+1} \right) \left( \hat{\bth}_{t+1}^{(j)} - \bar{\hat{\bth}}_{t+1}  \right) + \mathcal{O}\left( \left\| \hat{\bth}_{t+1}^{(j)} - \bar{\hat{\bth}}_{t+1}  \right\|^2 \right)   \nonumber \\
& = \nabla \bfor\left(\bar{\hat{\bth}}_{t+1} \right)   \left( \hat{\bth}_{t}^{(j)} - \bar{\hat{\bth}}_{t}  
 - K_t\left( H\hat{\bx}_t^{(j) }- H\bar{\hat{\bx}}_t  \right)  + K_t \left( \by_t^{(j)} - \bar{\by}  \right) \right) \nonumber \\
&~~~~ +\mathcal{O}\left( \left\| \hat{\bth}_{t+1}^{(j)} - \bar{\hat{\bth}}_{t+1}  \right\|^2 \right)  \;. \label{eq:x_deviation}
\end{align}
Multiplying Eq.~\eqref{eq:x_deviation} with $H$ yields the deviation of the reconstructed output with its ensemble mean:
\begin{align}
H\hat{\bx}_{t+1}^{(j)} - H \bar{\hat{\bx}}_{t+1} & = H \nabla \bfor\left(\bar{\hat{\bth}}_{t+1} \right)   \left( \hat{\bth}_{t}^{(j)} - \bar{\hat{\bth}}_{t}  
 - K_t\left( H\hat{\bx}_t^{(j) }- H\bar{\hat{\bx}}_t  \right)  + K_t \left( \by_t^{(j)} - \bar{\by}  \right) \right) \nonumber \\
&~~~~ +\mathcal{O}\left( \left\| \hat{\bth}_{t+1}^{(j)} - \bar{\hat{\bth}}_{t+1}  \right\|^2 \right)\;.  \label{eq:Hx_deviation}
\end{align}
Define a new state vector that includes the model parameters and the reconstructed outputs, centered with respect the corresponding means
\begin{equation}
\bz \doteq \left[ \bth - \bar{\bth}, H\bx - H\bar{\bx} \right]^T\;,
\end{equation}
and combine with Eqs.~\eqref{eq:theta_deviation}) and \eqref{eq:Hx_deviation} to yield the following update equation for $\bz$
\begin{equation}
\hat{\bz}_{t+1}^{(j)} = A_t \hat{\bz}_t^{(j)} + A_t \Delta_{\by,t}^{(j)} + \Lambda_{t}^{(j)}\;,\label{eq:z_update}
\end{equation}
where the evolution matrix
\begin{equation}
A_t\doteq
\begin{bmatrix}
     I               &    - K_t \\
      H \nabla \bfor\left(\bar{\hat{\bth}}_{t+1} \right)  & -H \nabla \bfor\left(\bar{\hat{\bth}}_{t+1} \right)  K_t
\end{bmatrix}\;,
\end{equation}
\begin{equation}
\Delta_{\by,t}^{(j)} \doteq
-
\begin{bmatrix}
    0   \\
   \by_t^{(j)} - \bar{\by}
\end{bmatrix}
, ~~
\mathcal{O}_{\bth,t}^{(j)} \doteq   \mathcal{O}\left( \left\| \hat{\bth}_{t+1}^{(j)} - \bar{\hat{\bth}}_{t+1}  \right\|^2 \right)
,~~
\Lambda_{t}^{(j)} \doteq
\begin{bmatrix}
    0    \\
  \mathcal{O}_{\bth,t}^{(j)} 
\end{bmatrix}\;.
\end{equation}
Therefore the evolution equation of the covariance matrix of $\hat{\bz}$ can be obtained from (\ref{eq:z_update})
\begin{align}
&C_{\hat{\bz}_{t+1},\hat{\bz}_{t+1}}  = A_t C_{\hat{\bz}_{t},\hat{\bz}_{t}}  A_t^T + A_t C_{\Delta_{\by,t}, \Delta_{\by,t}}A_t^T+ A_t C_{\hat{\bz}_{t},\Lambda_{t}} + C_{\Lambda_{t},\hat{\bz}_{t}}  A_t^T + C_{\Lambda_{t}, \Lambda_{t}}
\;, \label{eq:cov_update}
\end{align}
where
\begin{equation}
C_{\hat{\bz}_{t+1},\hat{\bz}_{t+1}} = 
\begin{bmatrix}
C_{\hat{\bth}_{t+1},\hat{\bth}_{t+1}} & C_{\hat{\bth}_{t+1},H\hat{\bx}_{t+1}} \\
C_{H\hat{\bx}_{t+1},\hat{\bth}_{t+1}} & C_{H\hat{\bx}_{t+1},H\hat{\bx}_{t+1}} \\
\end{bmatrix}\;, \label{eq:cov_z_matrix}
\end{equation}
\begin{equation}
C_{\Delta_{\by,t},\Delta_{\by,t}} = 
\begin{bmatrix}
0 & 0 \\
0 & C_{\by_t,\by_t} \\
\end{bmatrix}
= 
\begin{bmatrix}
0 & 0 \\
0 & \Gamma \\
\end{bmatrix}\;.
\end{equation}
Note that in obtaining Eq.~\eqref{eq:cov_update} we have used that the observations $\by$ are independent of the new state $\bz$, and thus $\Delta_{\by,t}$ is independent of $\hat{\bz}_t$ and $\Lambda_t$. 
To understand how different terms in Eq.~\eqref{eq:cov_z_matrix} evolve, Eq.~\eqref{eq:cov_update} is re-written in the matrix form 
\begin{align}
\begin{bmatrix}
C_{\hat{\bth}_{t+1},\hat{\bth}_{t+1}} & C_{\hat{\bth}_{t+1},H\hat{\bx}_{t+1}} \\
C_{H\hat{\bx}_{t+1},\hat{\bth}_{t+1}} & C_{H\hat{\bx}_{t+1},H\hat{\bx}_{t+1}} \\
\end{bmatrix}
& =
A_t
\begin{bmatrix}
C_{\hat{\bth}_{t},\hat{\bth}_{t}} & C_{\hat{\bth}_{t},H\hat{\bx}_{t}} \\
C_{H\hat{\bx}_{t},\hat{\bth}_{t}} & C_{H\hat{\bx}_{t},H\hat{\bx}_{t}} + \Gamma \\
\end{bmatrix}
A_t^T \nonumber \\
&
+ A_t 
\begin{bmatrix}
0 & C_{\hat{\bth}_{t},\Delta_{t}} \\
0 & C_{H\hat{\bx}_{t}, \Delta_{t}} \\
\end{bmatrix}
+ 
\begin{bmatrix}
0 & 0 \\
C_{\Delta_{t}, \hat{\bth}_t}  & C_{\Delta_{t}, H\hat{\bx}_{t}} \\
\end{bmatrix}
A_t^T  \nonumber \\
& + 
\begin{bmatrix}
0 & 0 \\
0 & C_{\Delta_{t}, \Delta_{t}}
\end{bmatrix}\;.
\label{eq:cov_update_matrix}
\end{align}
Clearly the properties of matrix $A_t$ determine the evolution of the covariances. Thus, we next consider the spectral properties of $A_t$. To reduce notational clutter, let 
\begin{equation}
A_t = 
\begin{bmatrix}
     I               &    - K_t \\
      H \nabla \bfor\left(\bar{\hat{\bth}}_{t+1} \right)  & -H \nabla \bfor\left(\bar{\hat{\bth}}_{t+1} \right)  K_t
\end{bmatrix}
\doteq \begin{bmatrix}
     I_{d_{\bth} \times d_{\bth}}              &    a \\
      b  & ba
\end{bmatrix}\;,
\end{equation}
and the characteristic equation for $A_t$ is given by 
\begin{align}
\det(A_t - \lambda I) &= \det\left((1-\lambda)I_{d_{\bth}\times d_{\bth}}\right) \det\left(ba -\lambda I_{d_y \times d_y}  - b \left((1-\lambda)I_{d_{\bth}\times d_{\bth}}\right)^{-1}a  \right) \nonumber \\
&= (1- \lambda)^{d_{\bth} - d_y} \lambda^{d_y} \det\left( \lambda I_{d_y\times d_y} - (I_{d_y\times d_y}  + ba)\right) = 0\;,
\end{align}
where $d_{\bth}$ and $d_{\by}$ are the dimensions for the model parameters $\bth$ and the observed outputs $\by$, respectively.
Solving the equation above yields the three different kinds of eigenvalues (the multiplicities are not necessarily equal to one):
\begin{align}
\lambda_1 &= 0\;, \\
\lambda_2 &= \lambda(I + ba) = \lambda\left(I - H\nabla \bfor\left( \bar{\hat{\bth}}_{t+1} \right) K_t \right) \approx \lambda\left( \Gamma \left( HC_{\hat{\bx}_t \hat{\bx}_t} H^T + \Gamma \right)^{-1} \right)\;, \\
\lambda_3 &= 1\;,
\end{align}
where we have used $\nabla \bfor(\bar{\hat{\bth}}_{t+1}) (\hat{\bth}_{t} - \bar{\hat{\bth}}_{t}) \approx \bfor(\hat{\bth}_{t}) - \bfor(\bar{\hat{\bth}}_{t})$. 
Note that $\lambda_2$ are the eigenvalues of the matrix $\Gamma \left( HC_{\hat{\bx}_t \hat{\bx}_t} H^T + \Gamma \right)^{-1}$, and therefore $0<\lambda_2 <1$. Additionally, $\lambda_3$ occurs for under-observed problems where the dimensionality of observations is less than the dimensionality of the hidden parameters ($d_{\by} < d_{\bth}$), which is common in practical applications.

\subsubsection{Convergence of $C_{\hat{\bth},\hat{\bth}}$ and $C_{\hat{\bth}, H\hat{\bx}}$} 
From Eq.\ (\ref{eq:cov_update_matrix}), the update equation for $C_{\bth,\bth}$ is
\begin{equation} \label{eq:c-theta-theta-update}
C_{\hat{\bth}_{t+1},\hat{\bth}_{t+1}} = C_{\hat{\bth}_{t},\hat{\bth}_{t}} - C_{\hat{\bth}_{t},H\hat{\bx}_{t}} \left(C_{H\hat{\bx}_{t},H\hat{\bx}_{t}}  + \Gamma \right)^{-1}C_{H\hat{\bx}_{t}, \hat{\bth}_t}\;.
\end{equation}
Note that $\hat{\bth}_t$ is a vector of dimension $d_{\bth}$. For each scalar entry of $\hat{\bth}_t$, the evolution equation is
\begin{equation}
C_{\hat{\bth}_{t+1}^s,\hat{\bth}^s_{t+1}} = C_{\hat{\bth}^s_{t},\hat{\bth}^s_{t}} - C_{\hat{\bth}^s_{t},H\hat{\bx}_{t}} \left(C_{H\hat{\bx}_{t},H\hat{\bx}_{t}}  + \Gamma \right)^{-1}C_{H\hat{\bx}_{t}, \hat{\bth}^s_t}\;,
\end{equation}
where the superscript $s$ denotes that $\hat{\bth}^s_{t}$ is scalar. Since the second term on the right hand side of the above equation is non-negative, the sequence $\left\{ C_{\hat{\bth}^s_{t},\hat{\bth}^s_{t}} \right\}_{t=1}^{\infty}$ is a monotone decreasing sequence for which all elements are bounded below ($\hat{\bth}^s_{t} \geq 0$). Thus the sequence is convergent based on monotone convergence theorem. This indicates that the incremental term must converge to zero, i.e.
\begin{equation}
\lim_{t\rightarrow \infty} C_{\hat{\bth}^s_{t},H\hat{\bx}_{t}} \left(C_{H\hat{\bx}_{t},H\hat{\bx}_{t}}  + \Gamma \right)^{-1}C_{H\hat{\bx}_{t}, \hat{\bth}^s_t} = 0
~~\Rightarrow~~
\lim_{t\rightarrow \infty}C_{\hat{\bth}^s_{t},H\hat{\bx}_{t}} = 0, \forall s = 1,..., d_{\bth}\;,
\end{equation}
which is equivalent to
\begin{equation}
\lim_{t\rightarrow \infty}C_{\hat{\bth}_{t},H\hat{\bx}_{t}} = 0\;. \label{eq:C_theta,Hx=0}
\end{equation} 
This further indicates the convergence of $C_{\hat{\bth}_t,\hat{\bth}_t}$ as 
\begin{equation}
\lim_{t\rightarrow \infty} \left\| C_{\hat{\bth}_{t+1}, \hat{\bth}_{t+1}} - C_{\hat{\bth}_t, \hat{\bth}_t} \right\|_F = 0\;,
\end{equation}
and the sequence of $C_{\hat{\bth}_t,\hat{\bth}_t}$ lies in a complete space (based on the Cauchy Convergence Criterion). However, this does not mean that the $C_{\hat{\bth}_t,\hat{\bth}_t}$ will converge to zero  (In real applications, we only require the covariance of $C_{\hat{\bth}, \hat{\bth}}$ to decrease below a preset tolerance for true convergence \cite{iglesias2013ensemble}). And in the case of $d_y < d_{\bth}$, $C_{\hat{\bth}_t,\hat{\bth}_t}$ will not converge to zero. 
This is related to the eigenvalue $\lambda_3 =1$ for  the evolution matrix $A_t$. To see this, the eigenvector of the matrix $A_t$ for $\lambda_3 = 1$ is solved,
\begin{align}
\left(A_t - \lambda_3 I \right) \bz  
= 
\begin{bmatrix}
     \bf{0}              &    a \\
      b  & ba - I
\end{bmatrix}
\begin{bmatrix}
    z_{\bth}     \\
      z_{Hx}   
\end{bmatrix}
 = 0
 \Rightarrow z_{Hx} = 0\;.
\end{align}
Therefore, the eigenvector $v_3$ corresponding to $\lambda_3 = 1$ lies in the subspace of $\bth$ and has zero component in the $Hx$ direction. 
The evolution matrix $A_t$ plays a key role in shrinking the covariances $C_{\hat{\bth}_t,\hat{\bth}_t}$, $C_{\hat{\bth}_t,H\hat{\bx}_t}$ and $C_{H\hat{\bx}_t, H\hat{\bx}_t}$ because $0<\lambda_1, \lambda_2<1$. However it does not have the variance shrinking effect in the direction of the eigenvector $v_3$ corresponding to $\lambda_3 = 1$ (violating the contraction requirement). Therefore, due to the non-shrinking direction $v_3$ lies in the subspace of $\bth$,  the sequence of $C_{\hat{\bth}_t,\hat{\bth}_t}$ will converge, but will not converge to $\bf{0}$.
\begin{align}
C_{\hat{\bth}_{t+1},H\hat{\bx}_{t+1}} = &~ C_{\hat{\bth}_{t},H\hat{\bx}_{t}} - C_{\hat{\bth}_{t},H\hat{\bx}_{t}} \left( C_{H\hat{\bx}_t, H\hat{\bx}_t} + \Gamma \right)^{-1} C_{H\hat{\bx}_t, H\hat{\bx}_t} \nonumber \\
&+ C_{\hat{\bth}_t, \Delta_t} - C_{\hat{\bth}_t, H\hat{\bx}_t} \left( C_{H\hat{\bx}_t, H\hat{\bx}_t} + \Gamma \right)^{-1} C_{H\hat{\bx}_t, \Delta_t}
\end{align}
\subsubsection{Convergence of $C_{H\hat{\bx}, H\hat{\bx}}$}
Based on Eq.\ (\ref{eq:cov_update_matrix}), the evolution equation for $C_{H\hat{\bx}, H\hat{\bx}}$ can be derived as 
\begin{align}
C_{H\hat{\bx}_{t+1},H\hat{\bx}_{t+1}} =&~ H \nabla \bfor\left(\bar{\hat{\bth}}_{t+1}  \right) 
\left(  C_{\hat{\bth}_{t},\hat{\bth}_{t}} - C_{\hat{\bth}_{t},H\hat{\bx}_{t}} \left(C_{H\hat{\bx}_{t},H\hat{\bx}_{t}}  + \Gamma \right)^{-1}C_{H\hat{\bx}_{t}, \hat{\bth}_t}\right) 
 \left( H \nabla \bfor\left(\bar{\hat{\bth}}_{t+1} \right) \right)^T \nonumber \\
& +  H \nabla \bfor\left(\bar{\hat{\bth}}_{t+1}  \right)  
\left(   C_{\hat{\bth}_t, \Delta_t} - C_{\hat{\bth}_t, H\hat{\bx}_t} \left( C_{H\hat{\bx}_t, H\hat{\bx}_t} + \Gamma \right)^{-1} C_{H\hat{\bx}_t, \Delta_t}   \right) \nonumber \\
& + \left(   C_{\Delta_t, \hat{\bth}_t} -  C_{\Delta_t, H\hat{\bx}_t} \left( C_{H\hat{\bx}_t, H\hat{\bx}_t} + \Gamma \right)^{-1} C_{H\hat{\bx}_t, \hat{\bth}_t}    \right)
\left( H \nabla \bfor\left(\bar{\hat{\bth}}_{t+1} \right) \right)^T  \nonumber \\
& + C_{\Delta_{t}, \Delta_{t}}\;,
\end{align}
and by applying Eq.\ (\ref{eq:theta_deviation}), the above equation can be simplified to
\begin{align} \label{eq:C_hxhx_reduce}
C_{H\hat{\bx}_{t+1},H\hat{\bx}_{t+1}} =&~ H \nabla \bfor\left(\bar{\hat{\bth}}_{t+1}  \right) 
C_{\hat{\bth}_{t+1}, \hat{\bth}_{t+1}} 
 \left( H \nabla \bfor\left(\bar{\hat{\bth}}_{t+1} \right) \right)^T \nonumber \\
& +  H \nabla \bfor\left(\bar{\hat{\bth}}_{t+1}  \right)  
C_{\hat{\bth}_{t+1}, \Delta_{t}} 
 + C_{\Delta_{t}, \hat{\bth}_{t+1}}
\left( H \nabla \bfor\left(\bar{\hat{\bth}}_{t+1} \right) \right)^T  \nonumber \\
& + C_{\Delta_{t}, \Delta_{t}} \;. 
\end{align}
Because $C_{\hat{\bth}_t, H\hat{\bx}_t}$ is equal to zero in the steady state based on Eq.\ (\ref{eq:C_theta,Hx=0}), a steady state condition of the nonlinear term $\Delta_t$ can be derived 
\begin{align}
C_{\hat{\bth}_t, H\hat{\bx}_t} &= E\left(\hat{\bth}_t - \bar{\hat{\bth}}_t, H\nabla \bfor\left(\bar{\hat{\bth}}_t \right) \left(\hat{\bth}_t - \bar{\hat{\bth}}_t\right)+ \Delta_{t-1} \right)  \nonumber \\
&= C_{\hat{\bth}_t, \hat{\bth}_t} \left(  H\nabla \bfor\left(\bar{\hat{\bth}}_t \right)  \right)^T + C_{\hat{\bth}_t, \Delta_{t-1}} = 0\;,
\end{align}
which is equivalent to 
\begin{align} \label{eq:stationary}
 C_{\hat{\bth}_t, \Delta_{t-1}} = - C_{\hat{\bth}_t, \hat{\bth}_t} \left(  H\nabla \bfor\left(\bar{\hat{\bth}}_t \right)  \right)^T \;.
\end{align}
This is the stationary condition for the nonlinear term.
Substituting the stationary condition (\ref{eq:stationary}) into Eq.\ (\ref{eq:C_hxhx_reduce}) yields 
\begin{align} \label{eq:Hx t+1 temp}
C_{H\hat{\bx}_{t+1},H\hat{\bx}_{t+1}} =&~ H \nabla \bfor\left(\bar{\hat{\bth}}_{t+1}  \right) 
C_{\hat{\bth}_{t+1}, \hat{\bth}_{t+1}} 
 \left( H \nabla \bfor\left(\bar{\hat{\bth}}_{t+1} \right) \right)^T \nonumber \\
& -  H \nabla \bfor\left(\bar{\hat{\bth}}_{t+1}  \right) 
C_{\hat{\bth}_{t+1}, \hat{\bth}_{t+1}} 
 \left( H \nabla \bfor\left(\bar{\hat{\bth}}_{t+1} \right) \right)^T
 - H \nabla \bfor\left(\bar{\hat{\bth}}_{t+1}  \right) 
C_{\hat{\bth}_{t+1}, \hat{\bth}_{t+1}} 
 \left( H \nabla \bfor\left(\bar{\hat{\bth}}_{t+1} \right) \right)^T  \nonumber \\
& + C_{\Delta_{t}, \Delta_{t}}  \nonumber \\
=&~ -H \nabla \bfor\left(\bar{\hat{\bth}}_{t+1}  \right) 
C_{\hat{\bth}_{t+1}, \hat{\bth}_{t+1}} 
 \left( H \nabla \bfor\left(\bar{\hat{\bth}}_{t+1} \right) \right)^T 
 + C_{\Delta_{t}, \Delta_{t}}\;.
\end{align}
The nonlinear term $\Delta_t  =  \mathcal{O}\left( \left\| \hat{\bth}_{t+1} - \bar{\hat{\bth}}_{t+1}  \right\|^2 \right)$ as a random variable can be decomposed into two parts: (a) a term that is correlated with $\hat{\bth}_{t+1}$ and (b) a term that is uncorrelated
\begin{equation}
\Delta_t = \beta^T \hat{\bth}_{t+1} + \gamma_{t+1}
\end{equation}
where $\beta$ can be obtained by using standard linear regression and applying Eq.\ (\ref{eq:stationary}):
\begin{equation}
\beta = C_{\hat{\bth}_{t+1},\hat{\bth}_{t+1}}^{-1} C_{\hat{\bth}_{t+1}, \Delta_t} = C_{\hat{\bth}_{t+1},\hat{\bth}_{t+1}}^{-1} C_{\hat{\bth}_{t+1}, \hat{\bth}_{t+1}} 
 \left( H \nabla \bfor\left(\bar{\hat{\bth}}_{t+1} \right) \right)^T =  -\left( H \nabla \bfor\left(\bar{\hat{\bth}}_{t+1} \right) \right)^T\;,
\end{equation}
and the covariance of $\Delta_t$ can be computed as 
\begin{align}
C_{\Delta_t, \Delta_t} &= \beta^T C_{\hat{\bth}_{t+1}, \hat{\bth}_{t+1}}\beta + C_{\gamma_{t+1}, \gamma_{t+1}} \nonumber \\
&= H \nabla \bfor\left(\bar{\hat{\bth}}_{t+1}  \right) 
C_{\hat{\bth}_{t+1}, \hat{\bth}_{t+1}} 
 \left( H \nabla \bfor\left(\bar{\hat{\bth}}_{t+1} \right) \right)^T
 + C_{\gamma_{t+1}, \gamma_{t+1}} 
\end{align}
Therefore, Eq.\ (\ref{eq:Hx t+1 temp}) can be reduced to the following simple form in the steady state.
\begin{equation}
C_{H\hat{\bx}_{t+1},H\hat{\bx}_{t+1}} =  C_{\gamma_{t+1}, \gamma_{t+1}} 
\end{equation}
where $\gamma_{t}$ represents the uncorrelated component of the nonlinear effect and it cannot be reduced by usual Kalman iterations.
Intuitively, this can also be directly observed from Eq.\ (\ref{eq:cov_update_matrix}) as the nonlinear term $C_{\Delta_{t}, \Delta_{t}}$ on the right hand side that is not shrunk by $A_t$ acts directly to the entry corresponding to $C_{H\hat{\bx}, H\hat{\bx}}$.

In conclusion, for the IEnKF, 
\begin{itemize}
\item $C_{\hat{\bth}_{t}, H\hat{\bx}_{t}} $ will converge to zero as $t$ increases.
\item $C_{H\hat{\bx}_{t}, H\hat{\bx}_{t}}$ will generally not converge to zero because of the nonlinearity effect of $\bfor(\bth)$. 
\item $C_{\hat{\bth}_{t}, \hat{\bth}_{t}}$ will generally not converge to zero because the insufficiency of observations.
\end{itemize} 

\subsubsection{Example: Evolution of covariances}
We return to the example in \S\ref{sec:example} to demonstrate the above conclusions. 
Figure~\ref{fig:evolution_cov_no_resampling} plot the evolution of the norms of the various covariance matrices, Kalman gain and innovation. 
The covariance matrix $C_{\hat{\bth}_{t}, H\hat{\bx}_{t}} $ converges to zero in relatively few iterations.
Consequently, $C_{\hat{\bth}_{t}, \hat{\bth}_{t}}$ stops updating.
$C_{H\hat{\bx}_{t}, H\hat{\bx}_{t}}$ oscillates and fails to converge due to the nonlinearity. 
The Kalman gain converges to zero, and the innovation fails to improve with subsequent iterations. This example demonstrated a typical situation that standard ensemble Kalman filter fails due to early stopping of the Kalman updates.
The numerical simulation results are consistent with the theoretical analysis above.

\begin{figure}[h]
  \centering
  \begin{subfigure}[t]{0.19\textwidth}
    \includegraphics[width=\textwidth]{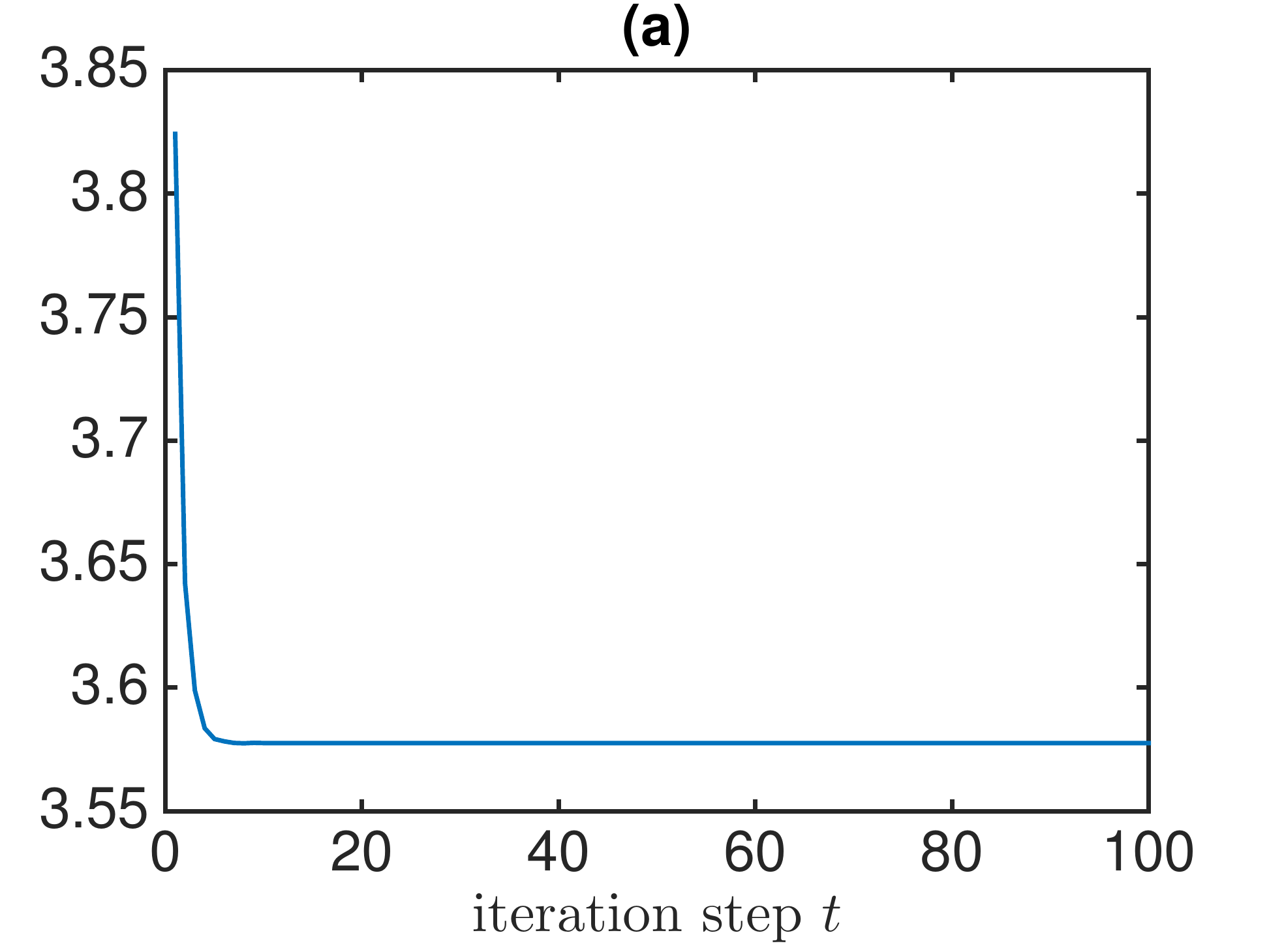}
    \caption{$\| C_{\hat{\bth}, \hat{\bth}} \|$}
  \end{subfigure}
  \begin{subfigure}[t]{0.19\textwidth}
    \includegraphics[width=\textwidth]{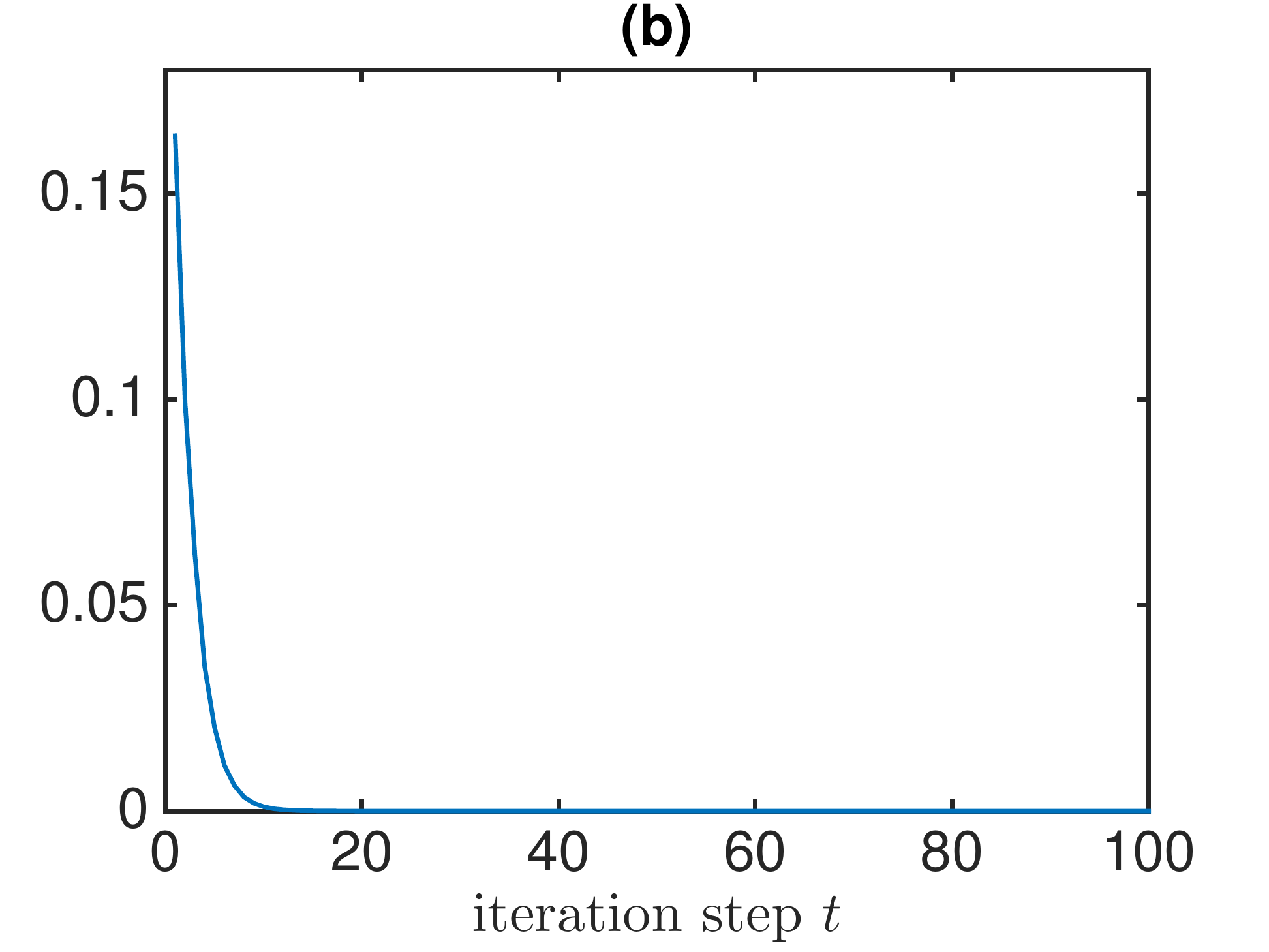}
    \caption{$\| C_{\hat{\bth}, H\hat{\bx}} \|$}
  \end{subfigure}
  \begin{subfigure}[t]{0.19\textwidth}
    \includegraphics[width=\textwidth]{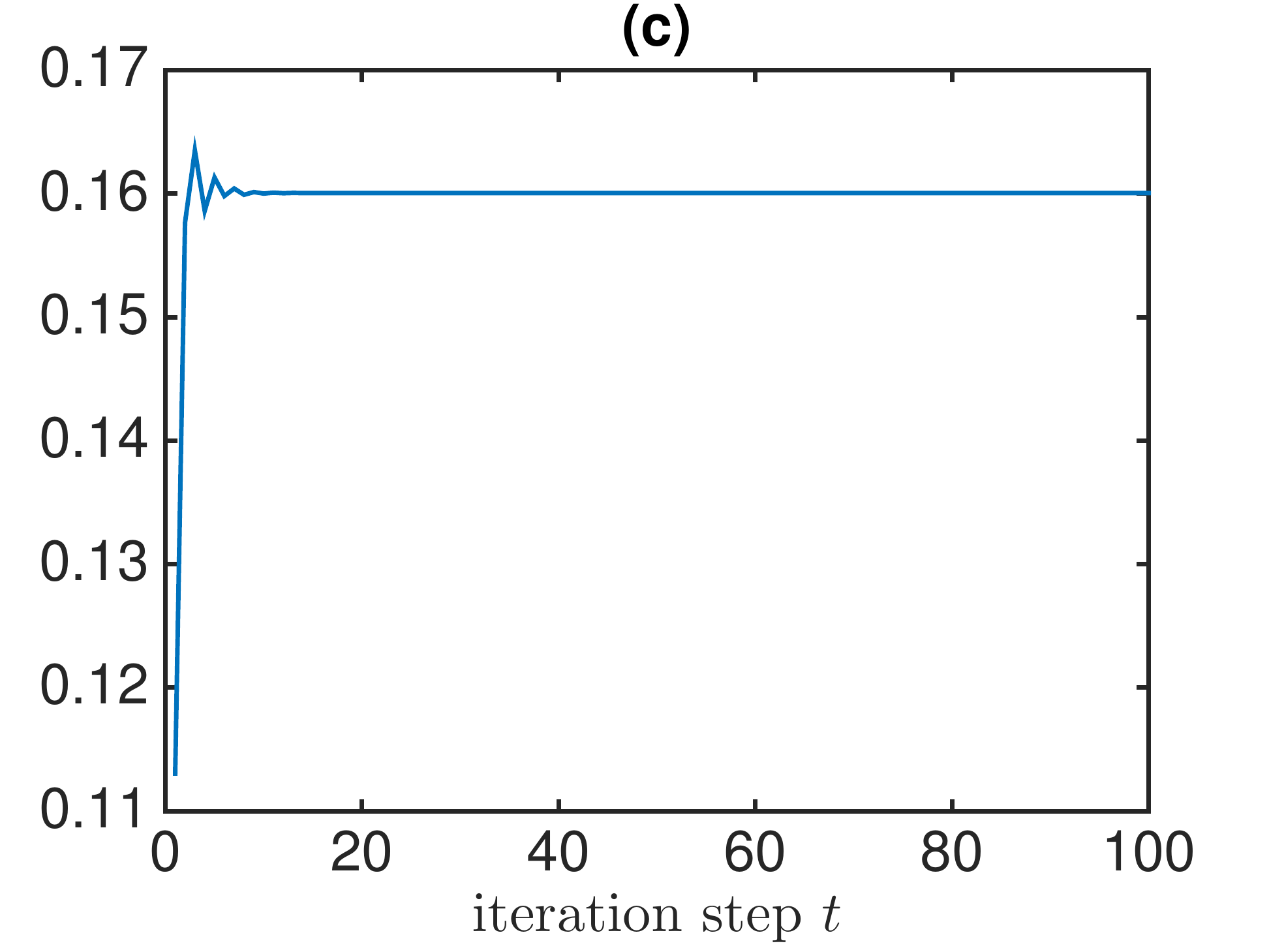}
    \caption{$\left\| C_{H\hat{\bx}, H\hat{\bx}} \right\|$}
  \end{subfigure}
  \begin{subfigure}[t]{0.19\textwidth}
    \includegraphics[width=\textwidth]{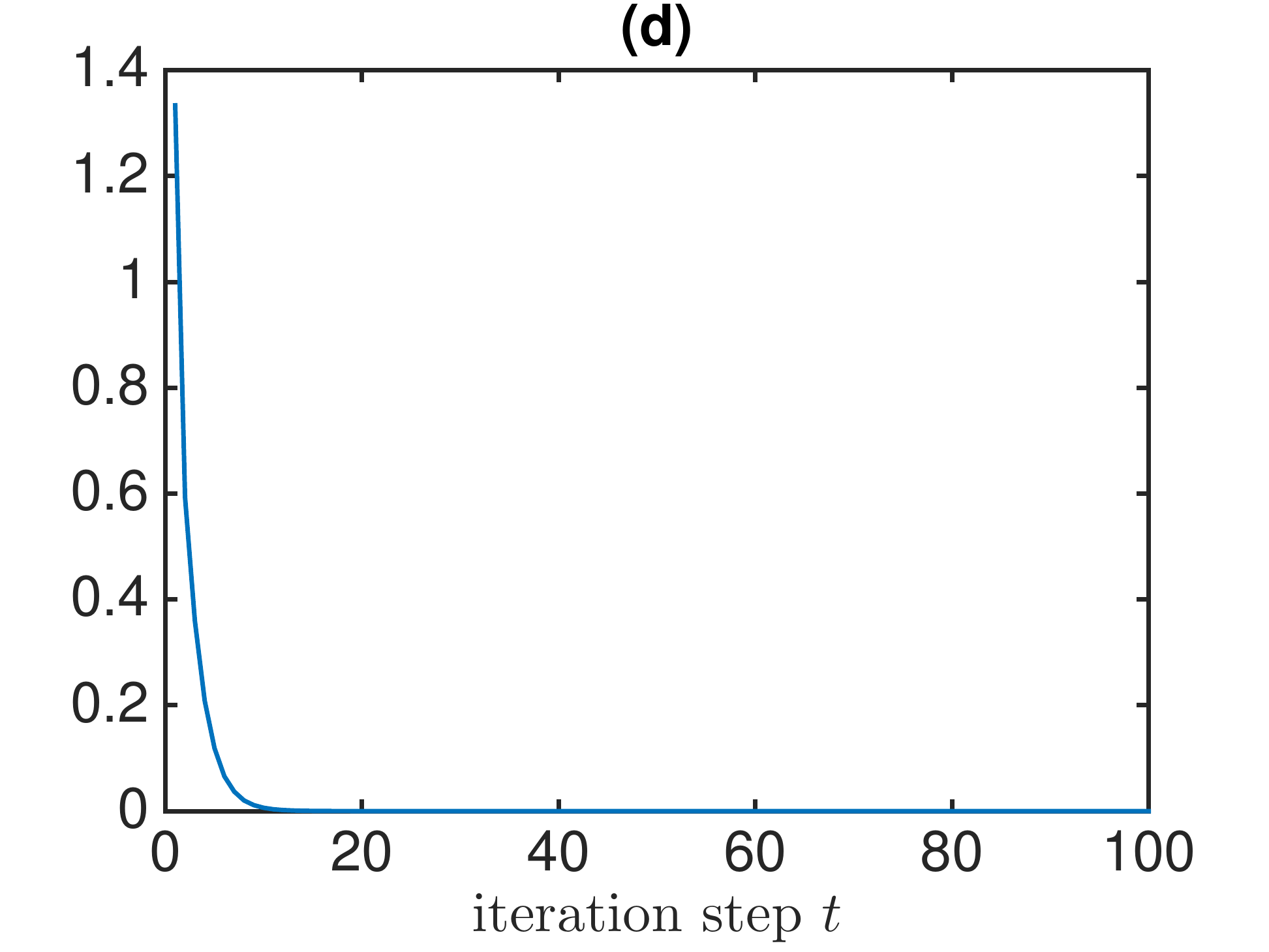}
    \caption{$\left\| K_t\right\|$}
  \end{subfigure}
   \begin{subfigure}[t]{0.19\textwidth}
    \includegraphics[width=\textwidth]{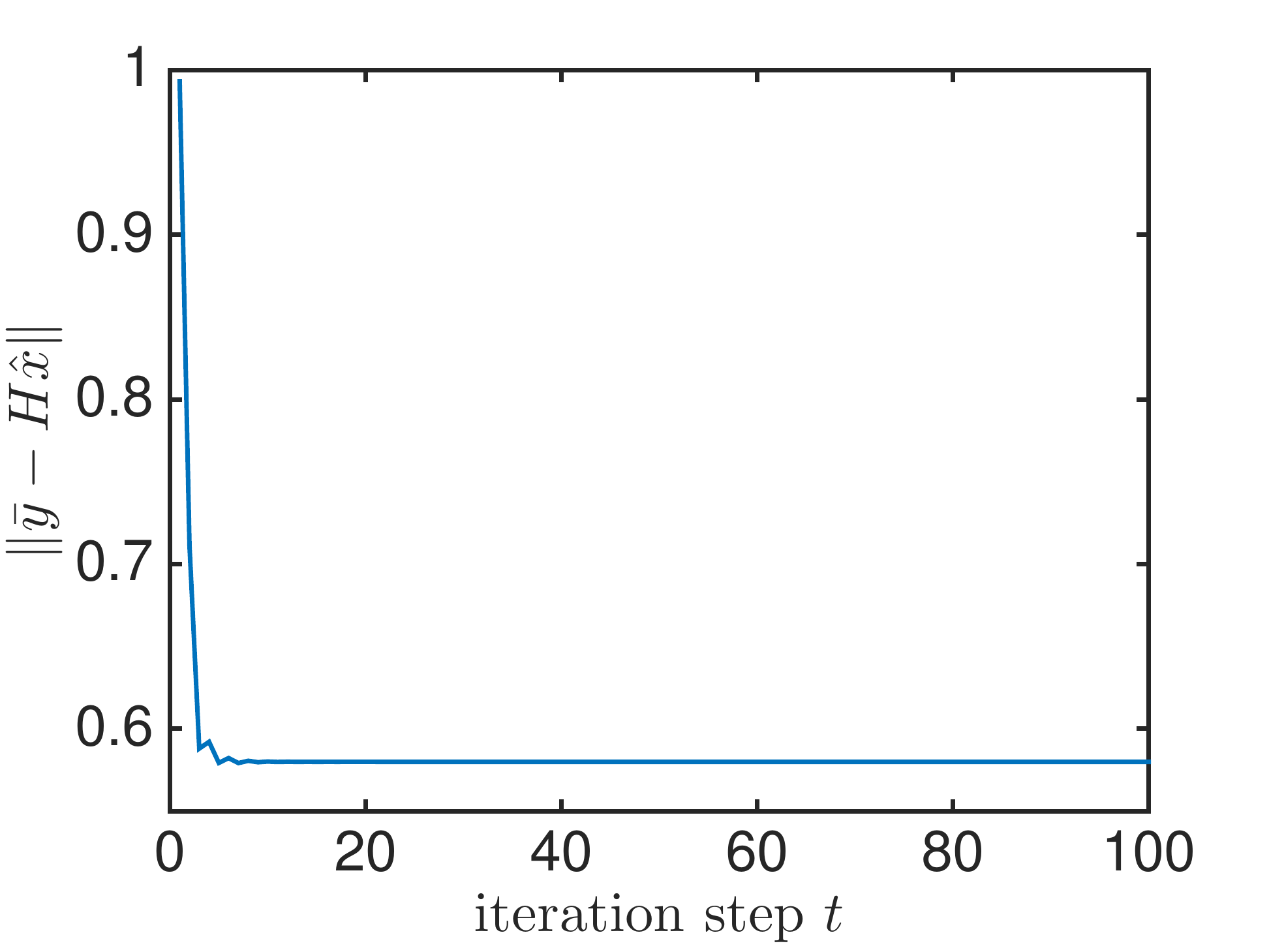}
    \caption{$\left\| \bar{\by} - H \bx_t \right\|$}
  \end{subfigure}
\caption{Evolution of the norms of the covariances, Kalman gain and innovation for the IEnKF applied to the example in \S\ref{sec:example}.}
\label{fig:evolution_cov_no_resampling}
\end{figure}

\subsection{Steady-state}
In \S\ref{sec:nonlinearity}, we focused on the transient behavior of the various covariance matrices and their relation to the ``early stopping" of the IEnKF in nonlinear problems. 
In this section, the steady-state behavior of the IEnKF is studied as well as the effect of observation uncertainty $\Gamma$.

Let $t=f$ denote the iteration number where $\bth_t$ is effectively unchanged with further iterations. 
Denote the (fixed) ensemble for $\bth$ at this step (and beyond) as
\begin{equation}
\left\{ \bth_f^{(j)} \right\}_{j=1}^J\;.
\end{equation}
There is no need to distinguish between the prior and posterior ensemble of $\bth$ as they are assumed the same for $t \geq f$. 
Therefore, the prior ensemble of the state $x$ will be 
\begin{equation}
\left\{ \hat{\bx}_f^{(j)} \right\}_{j=1}^J = \left\{ \bfor\left(\bth_f^{(j)}\right) \right\}_{j=1}^J\;,
\end{equation}
and will not change with respect to $t$ since $\left\{ \bth_f^{(j)} \right\}_{j=1}^J$ is fixed. 
The mean of the prior ensemble for $x$ is 
\begin{equation}
\bar{\hat{\bx}}_f  = \frac{1}{J} \sum _{j=1}^J \bfor\left(\bth_f^{(j)}\right)\;.
\end{equation}
The mean of the posterior ensemble of $x$ is 
\begin{align}
\bar{\bx}_f &= \frac{1}{J}\sum_{j=1}^J x_f^{(j)} \nonumber \\
& = \frac{1}{J}\sum_{j=1}^J \left( \hat{\bx}_f^{(j)} + C_{\hat{\bx}_f \hat{\bx}_f} H^T \left( HC_{\hat{\bx}_f \hat{\bx}_f} H^T + \Gamma \right)^{-1}\left( \by_t^{(j)} - H\hat{\bx}_f^{(j)} \right) \right)  \nonumber \\
& = \bar{\hat{\bx}}_f  + C_{\hat{\bx}_f \hat{\bx}_f} H^T \left( HC_{\hat{\bx}_f \hat{\bx}_f} H^T + \Gamma \right)^{-1}\left( \bar{\by} - H\bar{\hat{\bx}}_f \right)\;.
\end{align}
based on the relation between prior and posterior state estimation Eq.\ (\ref{eq:KF}b).
Therefore the final difference (i.e., oscillation magnitude) of the prior and posterior mean of the state $x$ is 
\begin{align}
\bar{\bx}_f - \bar{\hat{\bx}}_f = C_{\hat{\bx}_f \hat{\bx}_f} H^T \left( HC_{\hat{\bx}_f \hat{\bx}_f} H^T + \Gamma \right)^{-1}\left( \bar{\by} - H\bar{\hat{\bx}}_f \right)\;.
\end{align}
Left multiplying the above by $H$, 
\begin{align}
H\bar{\bx}_f - H\bar{\hat{\bx}}_f = H C_{\hat{\bx}_f \hat{\bx}_f} H^T \left( HC_{\hat{\bx}_f \hat{\bx}_f} H^T + \Gamma \right)^{-1}\left( \bar{\by} - H\bar{\hat{\bx}}_f \right)\;,
\label{eq:oscillation_magnitude}
\end{align}
gives the oscillation magnitude of the reconstructed output.
We can also obtain the relation between the output reconstruction error of the prior mean and posterior mean:
\begin{align} \label{eq:steady_state_error}
H\bar{\bx}_f - \bar{\by} = \Gamma \left( HC_{\hat{\bx}_f \hat{\bx}_f} H^T + \Gamma \right)^{-1}\left( H\bar{\hat{\bx}}_f - \bar{\by} \right)\;.
\end{align}
It can be seen from Eq.\ (\ref{eq:oscillation_magnitude}) and (\ref{eq:steady_state_error}) that the relative magnitudes of $\Gamma$ and $C_{H\hat{\bx}_f, H\hat{\bx}_f}$ determine the steady-state oscillation magnitude and mean output reconstruction error.
Based on the derivations above, the following summarizes the steady-state behavior of the standard IEnKF:
\begin{itemize}
\item The prior and posterior ensembles of the parameter $\bth$ will approach to a fixed ensemble $\left\{ \bth_f^{(j)} \right\}_{j=1}^J$.
\item The estimate for the system state $x$ will oscillate between a prior and posterior ensemble that both are fixed with respect to $t$. The oscillation magnitude of the prior and posterior means is given by (\ref{eq:oscillation_magnitude}).
\item The larger the uncertainty in the observation, $\Gamma$, the smaller the oscillation magnitude, $H\bar{\bx}_f - H\bar{\hat{\bx}}_f $;
\item The larger the uncertainty in the observation, $\Gamma$, the larger the final mean output reconstruction error, $H\bar{\bx}_f - \bar{\by}$. 
\end{itemize}


\subsection{Example}

Figure \ref{fig:compare_different_gamma} plots the convergence of the system state for different observation uncertainties $\Gamma = 0.1$ and $\Gamma = 0.0001$. It shows that the oscillation magnitude is smaller in the case of $\Gamma = 0.1$ than that for $\Gamma = 0.0001$. The intuition behind this is that when the observation uncertainty $\Gamma$ is smaller, the magnitude of the Kalman filter is larger. Therefore, the Kalman updates tend to be larger. It can also be shown that when $\Gamma$ is smaller, the posterior after each Kalman update is closer to the line of $\bar{\by} - Hx = 0$. This is because when the observation has less uncertainty, the reconstructed output $Hx$ is more likely to approach to the observation $\bar{\by}$.

\begin{figure}
  \centering
  \begin{tabular}{cc}
    \includegraphics[width=0.4\columnwidth]{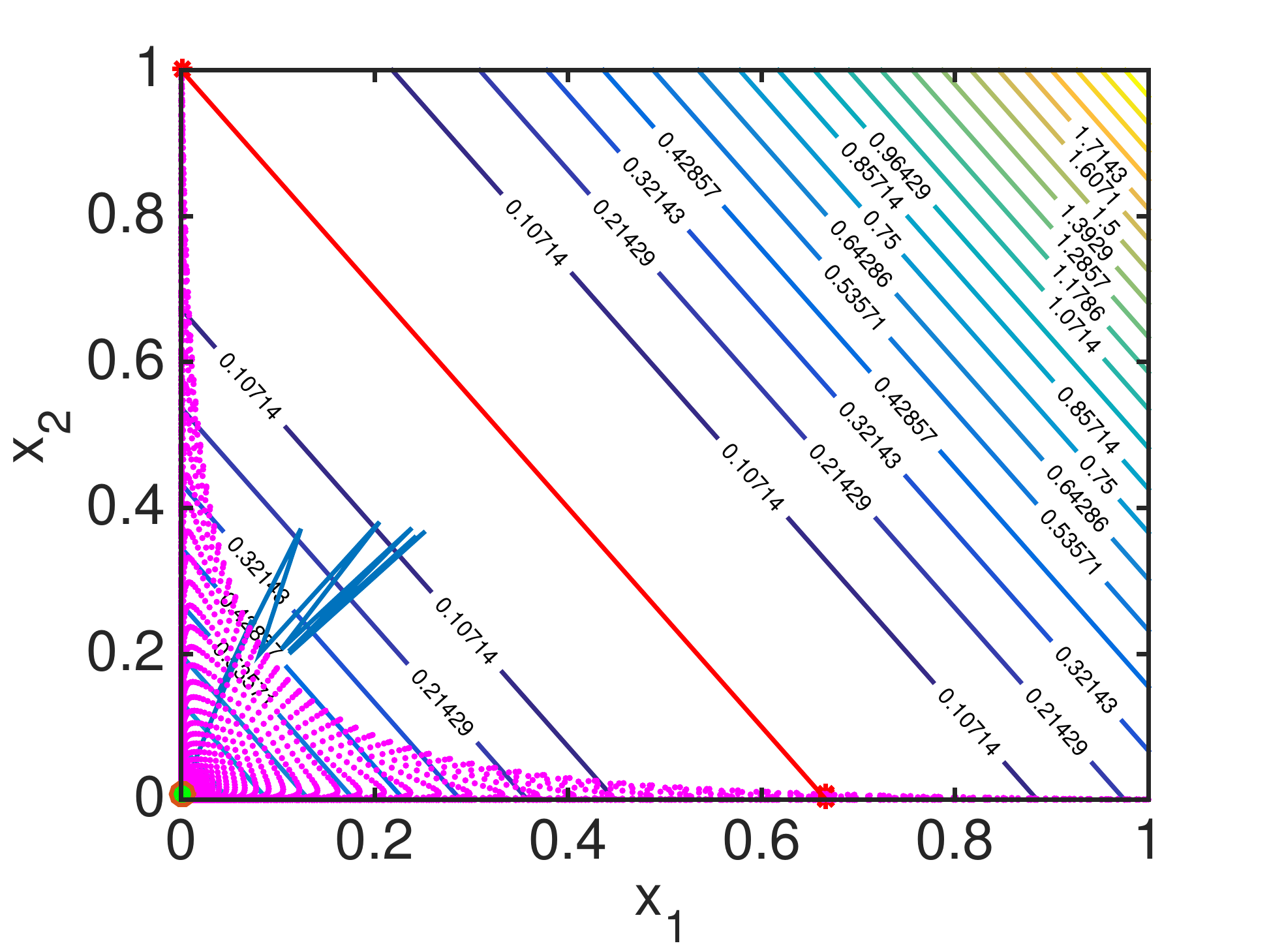}&
    
    \includegraphics[width=0.4\columnwidth]{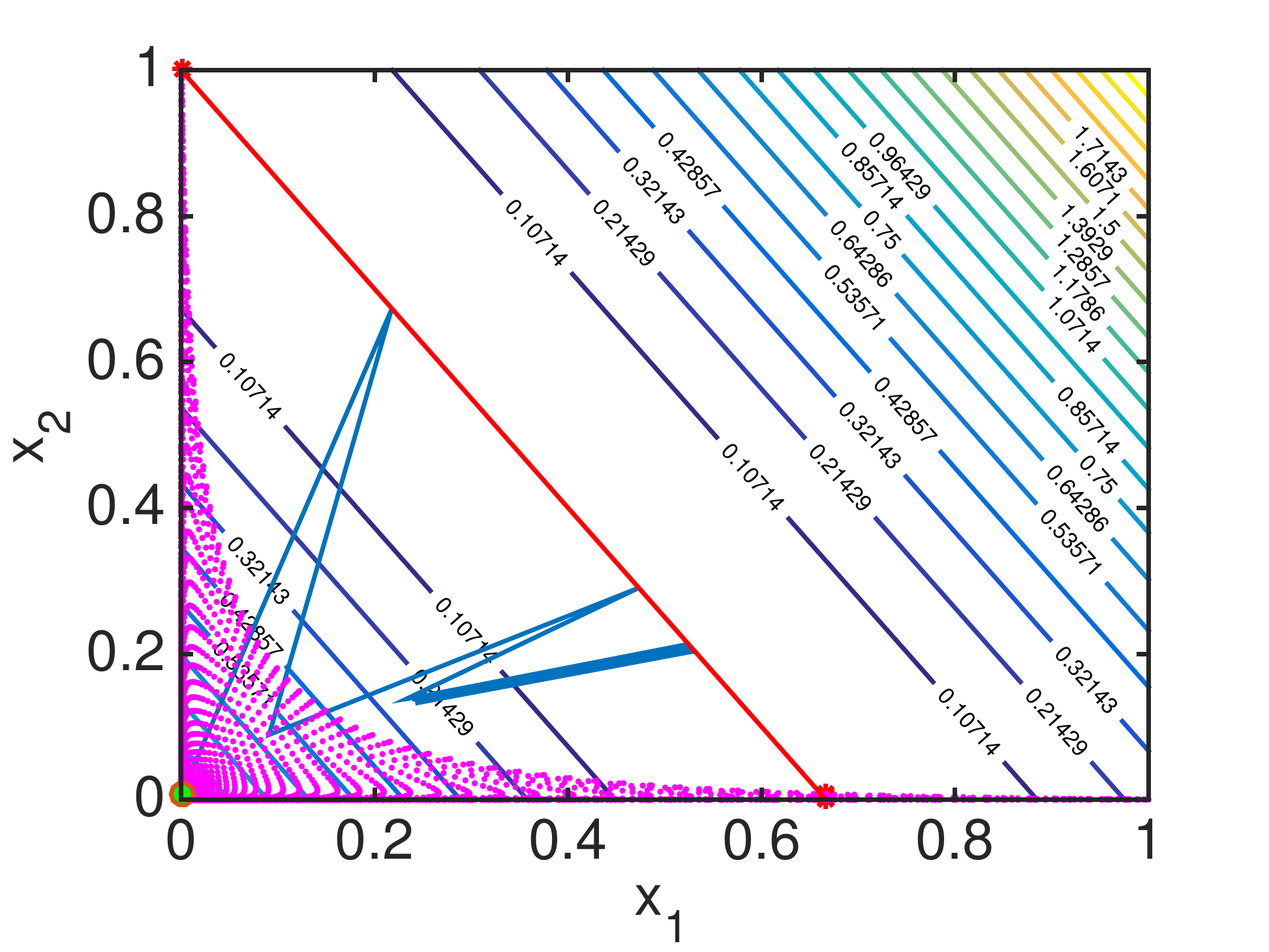}\\
  \end{tabular}
  \caption{(Left) the solution path of the ensemble mean of $x$ when $\Gamma = 0.1$. (Right) the solution path of the ensemble mean of $x$ when $\Gamma = 0.0001$.}
  \label{fig:compare_different_gamma}
  \end{figure}

\section{The Ensemble Resampling Method} \label{sec:resampling-method}
To prevent early stopping of the IEnKF we propose an {\em ensemble resampling method} whereby the prior ensemble members $\hat{\bth}_{t}^{(j)}$ and $\hat{\bx}_t^{(j)}$ are resampled before each Kalman update. For notational simplicity, we drop the hat $\hat{(\cdot)}$ and superscript $(j)$ unless needed. Namely, denote parameter and state values before resampling as random variables $\bth_{t}$ and $\bx_t$, and parameter and state values after resampling as random variables $\bth_{r,t}$ and $\bx_{r,t}$, with 
\begin{equation}
\bx_t = \bfor\left( \bth_t \right), ~~~\bx_{r,t} = \bfor\left( \bth_{r,t} \right)\;.
\end{equation}

The Kalman updates in Eq.~\eqref{eq:update-theta} and \eqref{eq:update-x} depends on the Kalman operators (Eq. \ref{eq:kalman-gain-theta} and \ref{eq:kalman-gain-x}), which effectively determine the ``update direction" in parameter/state space. These operators depend on the covariance matrices $C_{\hat{\bth}_t \hat{\bx}_t}$ and $C_{\hat{\bx}_t \hat{\bx}_t}$, which are updated each iteration based on the priors $\bth_t$ and $\bx_t$. Therefore, our strategy is to resample the prior ($\bth_t \leftarrow \bth_{r,t}$, $\bx_t \leftarrow \bx_{r,t}$) so that the covariance matrices (and hence Kalman operator) do not converging to zero before the innovation is minimized, while maintaining a consistent update direction with the traditional IEnKF strategy. We demonstrate in Thm.~\ref{Theorem:upper_bound} below that a consistent ``update direction'' is asymptotically maintained by ensuring the first and second moments of the prior distributions are maintained through the resampling process.   

\begin{theorem}
\label{Theorem:upper_bound}
Considering the IEnKF defined by the update step Eq.~\eqref{eq:KF} and the prediction step Eq.~\eqref{eq:forward}, if the first and second moment of the parameter $\bth$ are kept unchanged during the resampling process, i.e., 
\begin{align}
E\left(\bth_{r,t}\right) = E\left(\bth_{t}\right) = \bar{\bth}_t, ~~ Var\left(\bth_{r,t}\right) = Var\left(\bth_{t}\right) = \sigma_t^2\;. \nonumber 
\end{align}
then the deviations of the covariances have the following upper bounds
\begin{align}
\left\| C_{\bth_{r,t}, \bx_{r,t}} - C_{\bth_{t}, \bx_{t}} \right\| & \leq 2\sqrt{2} M\sigma_t^2\;,  \\
\left\| C_{\bx_{r,t}, \bx_{r,t}} - C_{\bx_{t}, \bx_{t}} \right\| & \leq 2\sqrt{2} M^2 \sigma_t^2\;, 
\end{align}
where the constant $M$ is a uniform upper bound for the gradient of the forward model mapping $\bfor(\bth)$, 
\begin{equation}
\left\| \nabla \bfor(\bth) \right\|\leq M\;.
\end{equation}
\end{theorem}
\begin{proof}
Assume the forward model $\bfor(\bth)$ is continuous on the closed interval $\left[\bth_t, \bar{\bth}_t\right]$ and differentiable on the open interval $\left(\bth_t, \bar{\bth}_t\right)$, where $\bth_t$ is an ensemble member at the step $t$ and $\bar{\bth_t}$ is the ensemble mean. Without loss of generality, assume $\bth_t< \bar{\bth}_t$. Based on the mean value theorem, 
\begin{align}
\bx_t = \bfor\left( \bth_t \right) = \bfor\left( \bar{\bth}_t \right) + \nabla \bfor (\bth_\eta) \left(  \bth_t -\bar{\bth}_t  \right)
\end{align}
where $\bth_\eta$ is some value in the open interval $\left(\bth_t, \bar{\bth}_t\right)$. Taking the expectation of the above equation yields
\begin{align}
\bar{\bx}_t = E\left( \bfor\left( \bth_t \right) \right) = \bfor\left( \bar{\bth}_t \right) + E\left(\nabla \bfor \bfor(\bth_\eta) \left(  \bth_t -\bar{\bth}_t  \right)\right)\;.
\end{align}
Therefore the deviation of the ensemble member from the ensemble mean is 
\begin{align}
\bx_t - \bar{\bx}_t = \nabla \bfor(\bth_\eta) \left(  \bth_t -\bar{\bth}_t  \right) - E\left(\nabla \bfor(\bth_\eta) \left(  \bth_t -\bar{\bth}_t  \right)\right)\;.
\end{align}
From the upper bound on $\nabla \bfor(\bth)$, we can obtain an upper bound on the variance of the state before resampling 
\begin{align}
E\left(    \bx_t - \bar{\bx}_t   \right)^2 & = E\left(  \nabla \bfor(\bth_\eta) \left(  \bth_t -\bar{\bth}_t  \right) - E\left(\nabla \bfor(\bth_\eta) \left(  \bth_t -\bar{\bth}_t  \right)\right) \right)^2 \nonumber \\
& = E\left(  \nabla \bfor(\bth_\eta) \left(  \bth_t -\bar{\bth}_t  \right)\right)^2 - \left(E\left(\nabla \bfor(\bth_\eta) \left(  \bth_t -\bar{\bth}_t  \right)\right) \right)^2 \nonumber \\
& \leq E\left(  \nabla \bfor(\bth_\eta) \left(  \bth_t -\bar{\bth}_t  \right)\right)^2 \nonumber \\
& \leq M^2 E\left(  \bth_t -\bar{\bth}_t  \right)^2 \nonumber \\
& = M^2 \sigma_t^2\;.
\label{eq:x_var}
\end{align}
Similarly, an upper bound of the variance of the state after resampling  can be obtained as
\begin{equation}
E\left( \bx_{r,t} - \bar{\bx}_{r,t} \right)^2 \leq M^2\sigma_t^2\;.
\label{eq:x_r}
\end{equation}
Because the ensemble members before and after resampling are independent, the mean square difference between $\bth_t$ and $\bth_{r,t}$ can be computed as 
\begin{align}
E\left(  \bth_{r,t} - \bth_t  \right)^2 &  = E\left(  \bth_{r,t} -\bar{\bth}_t + \bar{\bth}_t - \bth_t  \right)^2 \nonumber \\
& = E\left( \bth_{r,t} -\bar{\bth}_t \right)^2 + E\left( \bth_{t} -\bar{\bth}_t \right)^2 \nonumber \\
& = 2 \sigma_t^2\;.
\label{eq:theta_diff}
\end{align}
Applying the mean value theorem again, the difference between the states before and after resampling is
\begin{align}
\bx_{r,t} - \bx_t  = \nabla \bfor\left( \bth_\zeta \right)\left( \bth_{r,t} - \bth_t \right)\;,
\end{align}
where $\bth_\zeta \in \left( \bth_{r,t}, \bth_t \right)$. Again assume $ \bth_{r,t} < \bth_t$ without loss of generality. Taking expectation of the above equation, we can obtain an upper bound for the expected mean square difference between $ \bx_{r,t}$ and  $\bx_t $
\begin{align}
E\left( \bx_{r,t} - \bx_t \right)^2 & = E\left( \nabla \bfor\left( \bth_\zeta \right)\left( \bth_{r,t} - \bth_t \right) \right)^2 \nonumber \\
& \leq M^2 E\left(  \bth_{r,t} - \bth_t  \right)^2 \nonumber \\
& \leq 2 M^2 \sigma_t^2\;.
\label{eq:x_diff}
\end{align}
Based on (\ref{eq:x_var}), (\ref{eq:x_r}), (\ref{eq:theta_diff}) and (\ref{eq:x_diff}) and applying the Cauchy-Schwarz inequality, the upper bounds for the differences of covariances before and after resampling are
\begin{align}
& \left\| C_{\bth_{r,t}, \bx_{r,t}} - C_{\bth_{t}, \bx_{t}} \right\| \nonumber \\
& = \left\| C_{\bth_{r,t}, \bx_{r,t}} - C_{\bth_{r,t}, \bx_{t}} + C_{\bth_{r,t}, \bx_{t}}  - C_{\bth_{t}, \bx_{t}} \right\| \nonumber \\
& \leq \left\|   C_{\bth_{r,t}, \bx_{r,t}} - C_{\bth_{r,t}, \bx_{t}}  \right\| + \left\| C_{\bth_{r,t}, \bx_{t}}  - C_{\bth_{t}, \bx_{t}}  \right\| \nonumber \\
& = \left\| E\left( \bth_{r,t} - \bar{\bth}_t\right)\left(   \bx_{r,t} - \bx_t   \right) \right\|
+ \left\| E\left( \bth_{r,t} - {\bth}_t\right)\left(    \bx_t - \bar{\bx}_t   \right) \right\| \nonumber \\
&\leq \sqrt{E\left( \bth_{r,t} - \bar{\bth}_t\right)^2 E\left(   \bx_{r,t} - \bx_t   \right)^2} 
+ \sqrt{E\left( \bth_{r,t} - {\bth}_t\right)^2 E\left(    \bx_t - \bar{\bx}_t   \right)^2}  \nonumber \\
& \leq 2\sqrt{2} M\sigma_t^2\;,
\label{eq:u_bound_Cthetax}
\end{align}
\begin{align}
&\left\| C_{\bx_{r,t}, \bx_{r,t}} - C_{\bx_{t}, \bx_{t}} \right\| \nonumber \\
& = \left\| C_{\bx_{r,t}, \bx_{r,t}} - C_{\bx_{r,t}, \bx_{t}} + C_{\bx_{r,t}, \bx_{t}}  - C_{\bx_{t}, \bx_{t}} \right\| \nonumber \\
& \leq \left\|   C_{\bx_{r,t}, \bx_{r,t}} - C_{\bx_{r,t}, \bx_{t}}  \right\| + \left\| C_{\bx_{r,t}, \bx_{t}}  - C_{\bx_{t}, \bx_{t}}  \right\| \nonumber \\
& = \left\| E\left( \bx_{r,t} - \bar{\bx}_{r,t}\right)\left(   \bx_{r,t} - \bx_t   \right) \right\|
+ \left\| E\left( \bx_{r,t} - {\bx}_t\right)\left(    \bx_t - \bar{\bx}_t   \right) \right\| \nonumber \\
&\leq \sqrt{E\left( \bx_{r,t} - \bar{\bx}_{r,t}\right)^2 E\left(   \bx_{r,t} - \bx_t   \right)^2} 
+ \sqrt{E\left( \bx_{r,t} - {\bx}_t\right)^2 E\left(    \bx_t - \bar{\bx}_t   \right)^2}  \nonumber \\
& \leq 2\sqrt{2} M^2\sigma_t^2\;.
\label{eq:u_bound_Cxx}
\end{align}
\end{proof}

The above theorem establishes an upper bound for $\left\| C_{\bth_{r,t}, \bx_{r,t}} - C_{\bth_{t}, \bx_{t}} \right\|$ and $\left\| C_{\bx_{r,t}, \bx_{r,t}} - C_{\bx_{t}, \bx_{t}} \right\|$. However, for the IREnKF to be consistent with the IEnKF, we should demonstrate these differences converge to zero as $t \rightarrow \infty$. This will be established below by studying the behavior of $\sigma_t$. Moreover, to prevent early stopping of the IEnKF, we should demonstrate that $C_{\bth_{r,t}, Hx_{r,t}}$ does not converge to zero before the innovation is minimized. We next consider these two main requirements. 

We first note that the process of resampling can be equivalently viewed as adding a random deviation $\Delta_{r,t}^{(j)}$ to the $j$th posterior ensemble member before assigning it as a prior ensemble member for the next Kalman update
\begin{align}
\hat{\bth}_{t+1}^{(j)} & = \bth_{r,t}^{(j)} = \bth_t^{(j)} + \Delta_{r,t}^{(j)} 
 =\hat{\bth}_{t}^{(j)} + K_t\left( \by_t^{(j)} - H\hat{\bx}_t^{(j)} \right) 
+ \Delta_{r,t}^{(j)}\; ,
\label{eq:resample}
\end{align}
where $\Delta_{r,t}^{(j)} \doteq \bth_{r,t}^{(j)} -\bth_{t}^{(j)}$. Note that $E(\Delta_{r,t})=0$ and $Var\left( \Delta_{r,t} \right) = Var\left( \bth_{r,t} \right) + Var\left(\bth_{t} \right) = 2 \sigma_t^2$. 

Analogous with Eq.~\eqref{eq:z_update}, with resampling implemented, an evolution equation for the extended state $\bz$ can be derived as
\begin{equation}
\hat{\bz}_{t+1}^{(j)} = A_t \hat{\bz}_t^{(j)} + A_t \Delta_{\by,t}^{(j)} + \tilde{\Lambda}_{t}^{(j)}\;,\label{eq:z_update_resample}
\end{equation}
except the additional term $\tilde{\Lambda}_{t}^{(j)}$ now takes the form
\begin{equation}
\tilde{\Lambda}_{t}^{(j)} \doteq
\begin{bmatrix}
   \Delta_{r,t}^{(j)}    \\
   \tilde{\Delta}_t^{(j)}
\end{bmatrix}
,~~
\tilde{\Delta}_t^{(j)} : = \Delta_{t}^{(j)} + H \nabla \bfor\left(\bar{\hat{\bth}}_{t+1} \right) \Delta_{r,t}^{(j)} \;,
\end{equation}
which includes both the effects from the nonlinearity, $\Delta_{t}^{(j)}$,  and resampling, $\Delta_{r,t}^{(j)}$. 
And similar to Eq.~\eqref{eq:cov_update_matrix}, it can be shown that now with resampling,
\begin{align} \label{eq:resample_cov_update}
\begin{bmatrix}
C_{\hat{\bth}_{t+1},\hat{\bth}_{t+1}} & C_{\hat{\bth}_{t+1},H\hat{\bx}_{t+1}} \\
C_{H\hat{\bx}_{t+1},\hat{\bth}_{t+1}} & C_{H\hat{\bx}_{t+1},H\hat{\bx}_{t+1}} \\
\end{bmatrix} 
&=
A_t
\begin{bmatrix}
C_{\hat{\bth}_{t},\hat{\bth}_{t}} & C_{\hat{\bth}_{t},H\hat{\bx}_{t}} \\
C_{H\hat{\bx}_{t},\hat{\bth}_{t}} & C_{H\hat{\bx}_{t},H\hat{\bx}_{t}} + \Gamma \\
\end{bmatrix}
A_t^T \nonumber \\
&
+ A_t 
\begin{bmatrix}
C_{\hat{\bth}_{t},{\Delta}_{r,t}} & C_{\hat{\bth}_{t},\tilde{\Delta}_{t}} \\
C_{H\hat{\bx}_{t}, {\Delta}_{r,t}} & C_{H\hat{\bx}_{t}, \tilde{\Delta}_{t}} \\
\end{bmatrix}
+ 
\begin{bmatrix}
C_{{\Delta}_{r,t},\hat{\bth}_{t}} & C_{{\Delta}_{r,t},H\hat{\bx}_{t}} \\
C_{\tilde{\Delta}_{t}, \hat{\bth}_t}  & C_{\tilde{\Delta}_{t}, H\hat{\bx}_{t}} \\
\end{bmatrix}
A_t^T  \nonumber \\
& + 
\begin{bmatrix}
C_{{\Delta}_{r,t}, {\Delta}_{r,t}} & C_{{\Delta}_{r,t}, \tilde{\Delta}_{t}} \\
C_{\tilde{\Delta}_{t}, {\Delta}_{r,t}} & C_{\tilde{\Delta}_{t}, \tilde{\Delta}_{t}}
\end{bmatrix}\;.
\end{align}
In contrast to Eq.~\eqref{eq:cov_update_matrix}, due to the deviation caused by resampling, the fourth term has non-zero contributions to all four entries of the covariance matrix. Therefore, $C_{\hat{\bth}, \hat{\bth}}$, $C_{\hat{\bth}, H\hat{\bx}}$ or $C_{H\hat{\bx}, H\hat{\bx}}$ will not converge to zero unless the nonlinear effect $\Delta_t$ and the resampling deviation $\Delta_{r,t}$ both approach zero. In turn, this prevents the Kalman operator $K_t$ from approaching to zero  before the innovation $\bar{\by} - H\hat{\bx}$ approaches to zero.

We next consider the asymptotic behavior of the covariance matrices. Similar to Eq.~\ref{eq:c-theta-theta-update}, a covariance update equation for $\bth$ can be obtained as 
\begin{align}
C_{\hat{\bth}_{r, t+1}, \hat{\bth}_{r, t+1}}
 = C_{\hat{\bth}_{r,t}, \hat{\bth}_{r,t}} - C_{\hat{\bth}_{r,t}, H\hat{\bx}_{r,t}} \left( C_{H\hat{\bx}_{r,t}, H\hat{\bx}_{r,t}} + \Gamma \right)^{-1}C_{H \hat{\bx}_{r,t}, \hat{\bth}_{r,t}}\;,
\end{align}
For the case where $\bth$ and $H\bx$ are scalar, the above relation reduces to
\begin{align}\label{eq:scalar_sigma_evolve}
\sigma_{t+1}^2 = \sigma_{t}^2 - \frac{C_{\hat{\bth}_{r,t}, H\hat{\bx}_{r,t}}C_{H \hat{\bx}_{r,t}, \hat{\bth}_{r,t}}}{C_{H\hat{\bx}_{r,t}, H\hat{\bx}_{r,t}} + \Gamma}\;.
\end{align}
Because $\sigma^2$ is always non-negative and the second term on the right is non-negative, from the monotone convergence theorem, it can be shown that the sequence $\left\{ \sigma_t \right\}_{t=1}^\infty$ is convergent, and we will denote the limit as $\sigma_*$. 
Also, since $\left\{ \sigma_t \right\}_{t=1}^\infty$ is convergent, the following equation also holds true:
\begin{equation}
\lim_{t\rightarrow \infty}\left| C_{\hat{\bth}_{r,t}, H\hat{\bx}_{r,t}} \right| = 0\;.
\end{equation}
There are two reasons for the above covariance $C_{\hat{\bth}_{r,t}, H\hat{\bx}_{r,t}} $ to converge to zero:
\begin{itemize}
\item The ensembles $\left\{ \hat{\bth}_{t}^{(j)} \right\}_{j=1}^J$ and $\left\{ H\hat{\bx}_{t}^{(j)} \right\}_{j=1}^J$ becomes uncorrelated while the ensembles do not collapse to a single point.
\item The ensemble $\left\{ \hat{\bth}_{t}^{(j)} \right\}_{j=1}^J$ collapses to a single point (thus $\left\{ H\hat{\bx}_{t}^{(j)} \right\}_{j=1}^J$ also collapses to a single point), and therefore $C_{\hat{\bth}_{r,t}, H\hat{\bx}_{r,t}} $ converges to zero. 
\end{itemize}
The first case corresponds to the false convergence of IEnKF and the second case corresponds to the true convergence~\cite{schillings2017convergence} (note that this only denotes the convergence of IEnKF and does not guarantee convergence to the true global minimum in a non-convex setting). Both cases are possible in numerical simulations of IEnKF. 
However, because of the additional random disturbance added by resampling $\Delta_{r,t}$, the covariance $C_{\hat{\bth}_{r,t}, H\hat{\bx}_{r,t}}$ will not be equal to zero unless $\Delta_{r,t} =0$ (i.e.\ $\sigma_t = 0$). 
Mathematically, this can be formulated as the following conjecture.
\begin{conjecture}
There exists a non-negative function $g(\sigma_t)\geq 0$ such that
\begin{align}
\left| C_{\hat{\bth}_{r,t}, H\hat{\bx}_{r,t}} \right| \geq g\left(\sigma_t \right) >0,~\forall \sigma_t \neq 0, 
\end{align}
and $g\left(\sigma_t \right) = 0$ only when $\sigma_t = 0$. 
\end{conjecture}
Based on the above conjecture, Eq.\ (\ref{eq:scalar_sigma_evolve}) can be converted to 
\begin{align}
\sigma_{t+1}^2 \leq \sigma_t^2 - cg^2\left( \sigma_t \right)\;,
\end{align}
where $c>0$ is a constant that accounts for the denominator of the second term on the right side of Eq.~\eqref{eq:scalar_sigma_evolve}. Taking the limit $t \rightarrow \infty$ on both sides of the above inequality yields
\begin{align}
\sigma_*^2 \leq \sigma_*^2 - cg^2(\sigma_*) \Rightarrow g^2\left( \sigma_* \right) \leq 0 \Rightarrow g\left( \sigma_* \right) = 0\;.
\end{align}
Because zero is the only stationary point of $g(\cdot)$, it follows that $\sigma_* = 0$, i.e.\
\begin{equation}
\lim_{t\rightarrow \infty} \sigma_t^2 = 0\;.
\label{eq:sigam_theta_converge_zero}
\end{equation}
Therefore, from Eq.~\eqref{eq:u_bound_Cthetax}, as the iteration step $t$ increases, the difference between $C_{\bth_{r,t}, \bx_{r,t}}$ and $C_{\bth_{t}, \bx_{t}}$ will asymptotically converge to zero. This result also holds for the difference between $C_{\bx_{r,t}, \bx_{r,t}}$ and $C_{\bx_{t}, \bx_{t}} $.
Therefore, except at the early stage of the iteration process (where resampling prevents $C_{\hat{\bth}, H\hat{\bx}}$ from converging to zero), the change of the Kalman gains will be small and asymptotically converge to zero if the mean and the covariance of the parameter ensemble are kept the same before and after the resampling process.

\subsection{Resampling applied to example} 
To verify the conclusions above, proposed IREnKF is applied to the example problem defined in Eq.~\eqref{eq:F_example} and \eqref{eq:H_example}. 
A Gaussian distribution was used as the resampling distribution. 
Figure \ref{fig:evolution_cov_with_resampling} shows the evolution of the covariances with resampling implemented. 
Compared to the results without resampling, the covariance $C_{\hat{\bth}, H\hat{\bx}}$ and the Kalman gain $K_t$ do not approach to zero in the early stage to cause $C_{\hat{\bth}, \hat{\bth}}$ and $C_{H\hat{\bx}, H\hat{\bx}}$ to stop updating. 
Instead, all covariances converge to zero simultaneously and roughly at the same time that the magnitude of the innovation $\left\| \bar{\by} - H\hat{\bx} \right\|$ converges to zero (see  Fig.~\ref{fig:fig_misfit_with_resampling}). 
This indicates that the ``early stopping" of the standard IEnKF is prevented  by resampling as compared to the simulations without resampling where the Kalman gain $K_t$ converges to zero before the innovation is minimized (i.e. converges to zero). 

Figure \ref{fig:evolution_cov_with_resampling}(a) also shows that $C_{\hat{\bth},\hat{\bth}}$ converges to zero as $t$ increases. This verifies the conclusion from Eq.~\eqref{eq:sigam_theta_converge_zero}, and along with Thm~\ref{Theorem:upper_bound}, leads to the conclusion that the change of the Kalman gain after resampling will remain small and asymptotically converge to zero. This conclusion is numerically verified by results shown in Fig.~\ref{fig:fig_diff_K_t}, which demonstrates that the difference of the Kalman gain before and after resampling quickly decreases, remains bounded by a small quantity, and ultimately converges to zero as $t$ increases.

Figure \ref{fig:fig_theta_convergence_with_resampling} and Figure \ref{fig:fig_x_convergence_with_resampling} show the solution path with resampling implemented in parameter $\bth$ and state $x$ space, respectively. Both figures show that the ``early stopping" of IEnKF is prevented and the solutions are able to converge to the true solution. Note that there are multiple local minima for this test problem and the solution only converges to one of them depending on the initial condition given.

\begin{figure}[h]
  \centering
  \begin{tabular}{cc}
    \includegraphics[width=0.4\columnwidth]{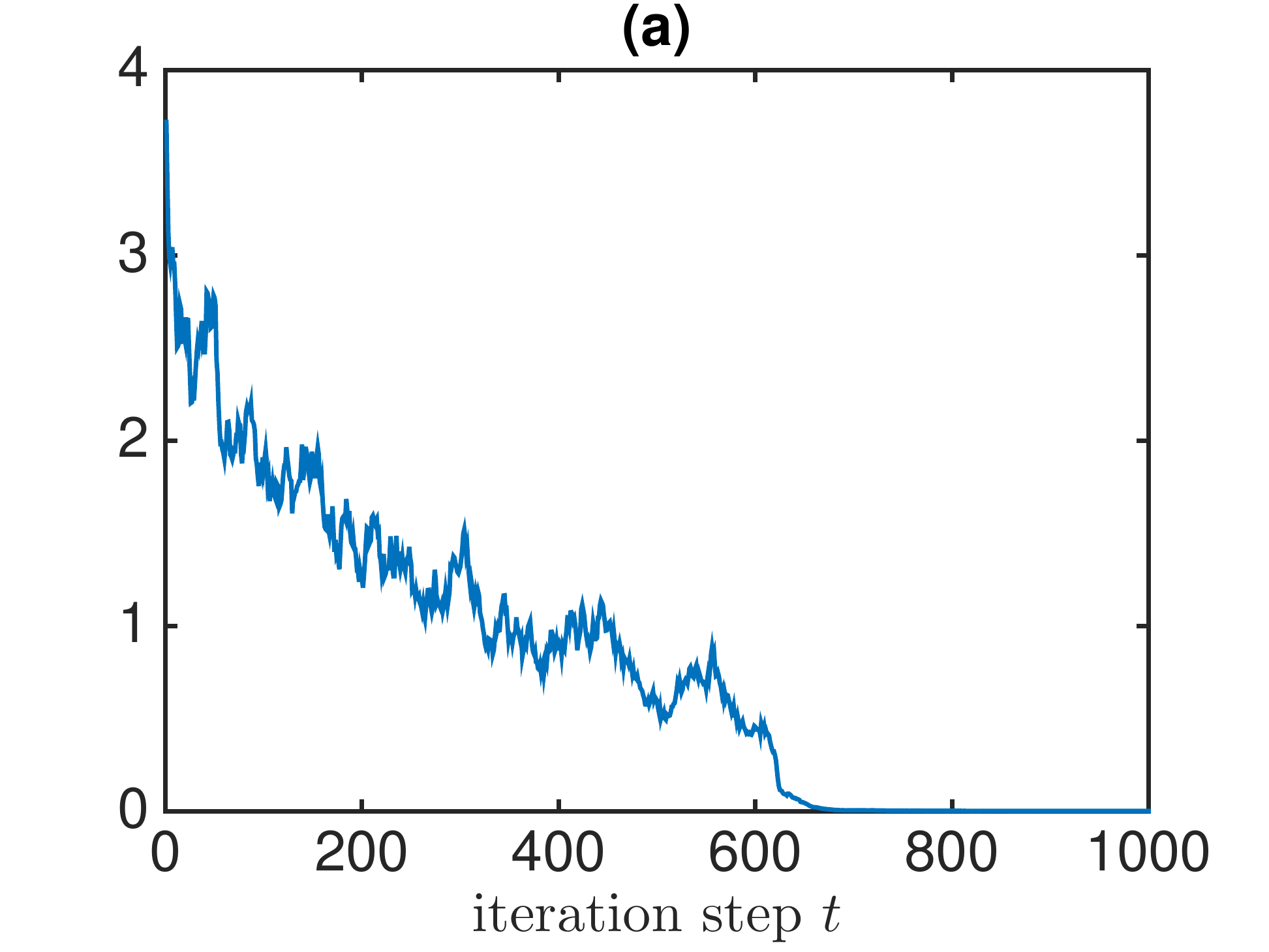}&
    
    \includegraphics[width=0.4\columnwidth]{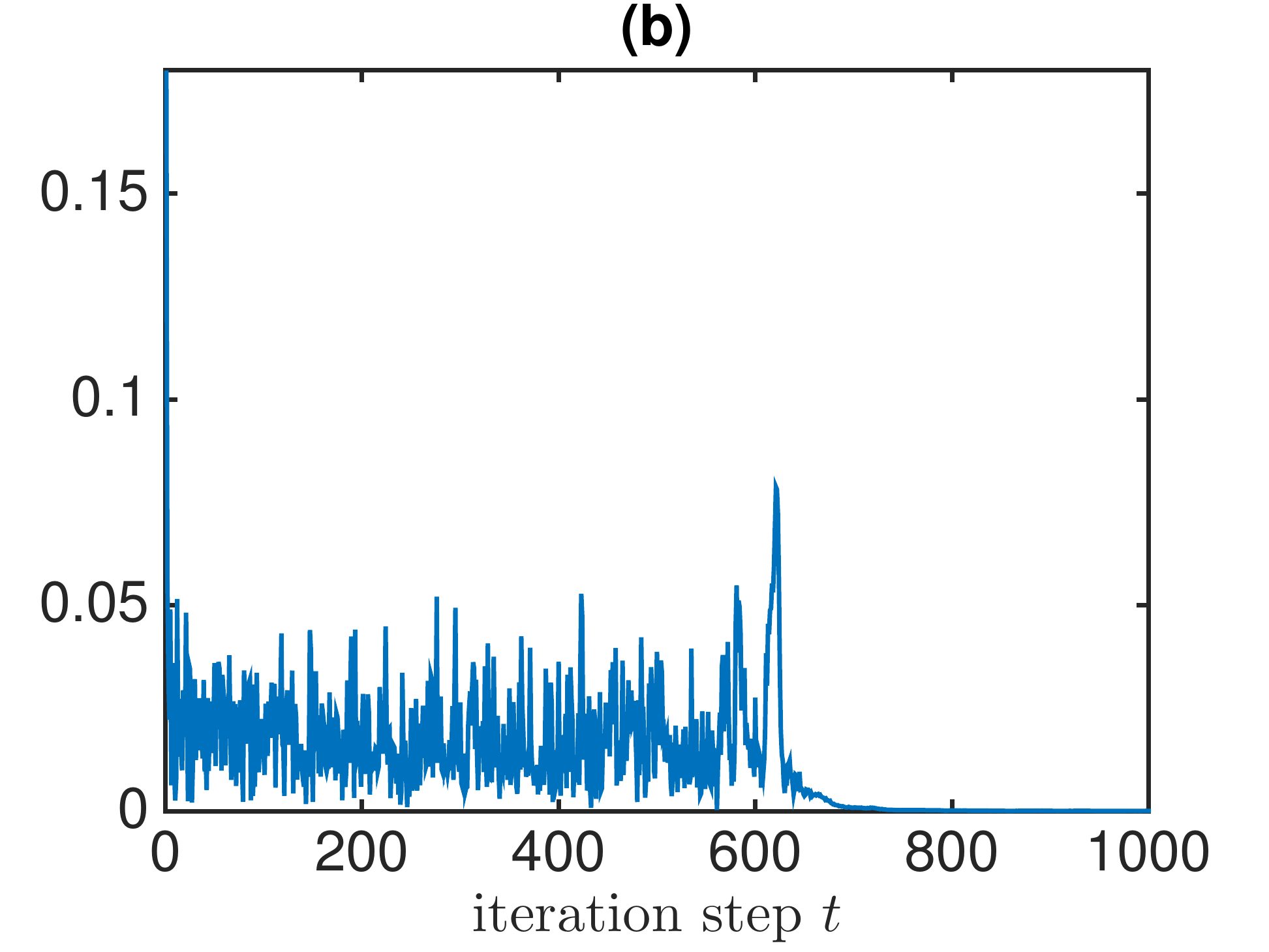}\\
    
    \includegraphics[width=0.4\columnwidth]{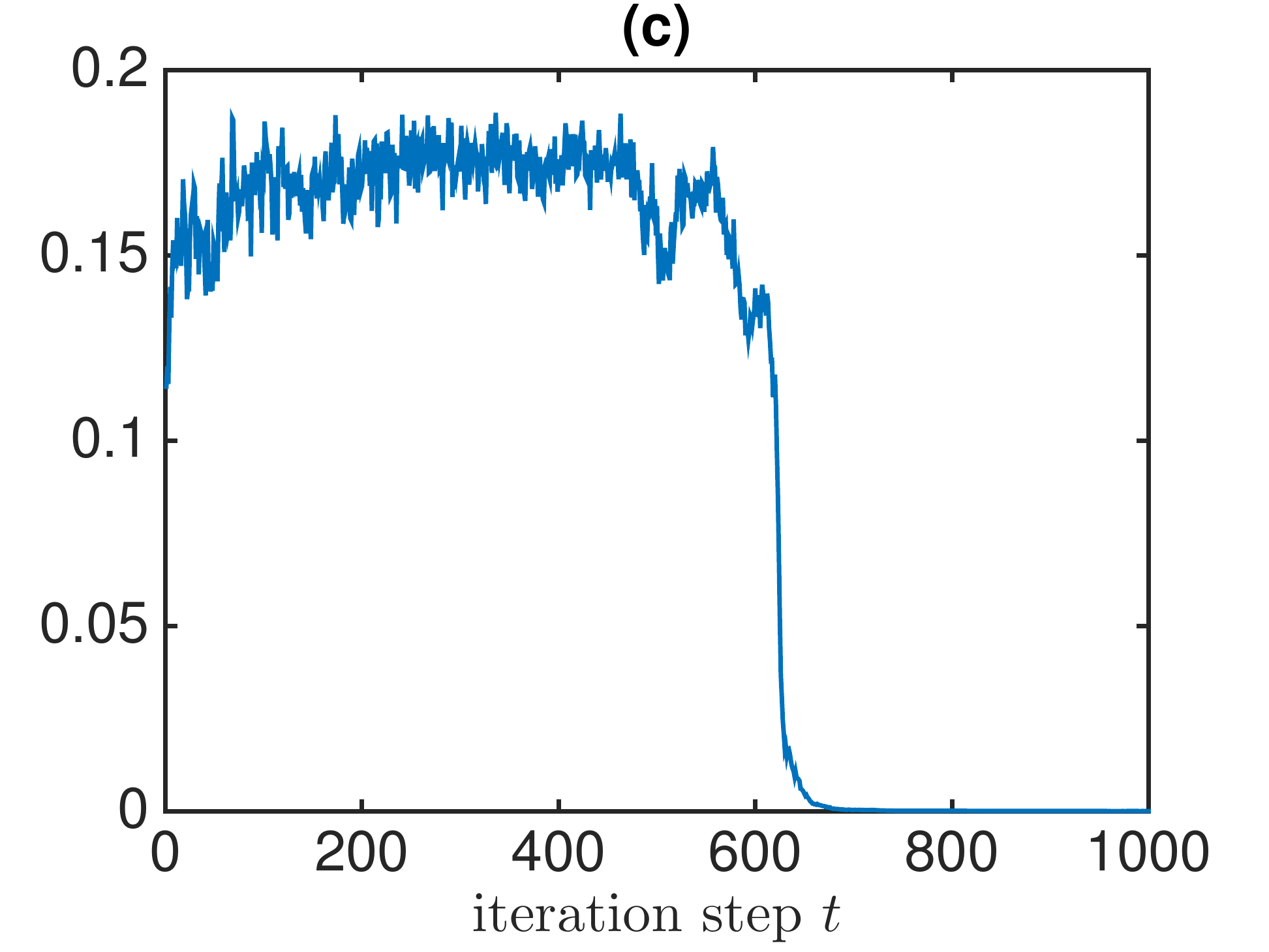}&
    
    \includegraphics[width=0.4\columnwidth]{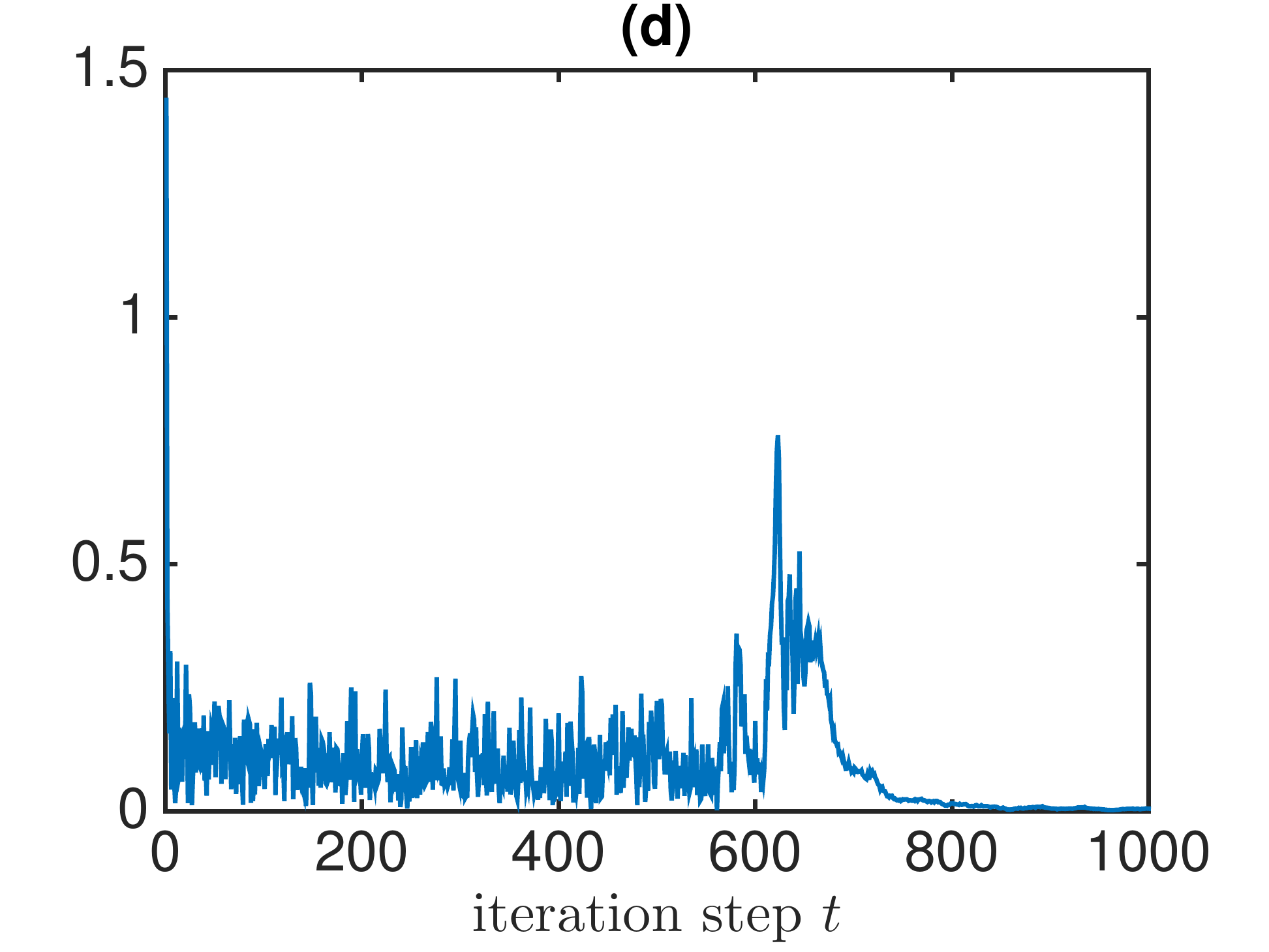}\\
  \end{tabular}
  \caption{The evolution of the norms of the covariances and the Kalman gain with resampling implemented.
  (a): $\left\|C_{\hat{\bth}, \hat{\bth}}\right\|$, (b): $\left\| C_{\hat{\bth}, H\hat{\bx}}\right\|$, (c): $\left\|C_{H\hat{\bx}, H\hat{\bx}}\right\|$, (d): $\left\| K_t\right\|$ }
\label{fig:evolution_cov_with_resampling}
\end{figure}

\begin{figure}[h]
\centering
\includegraphics[width=0.45\columnwidth]{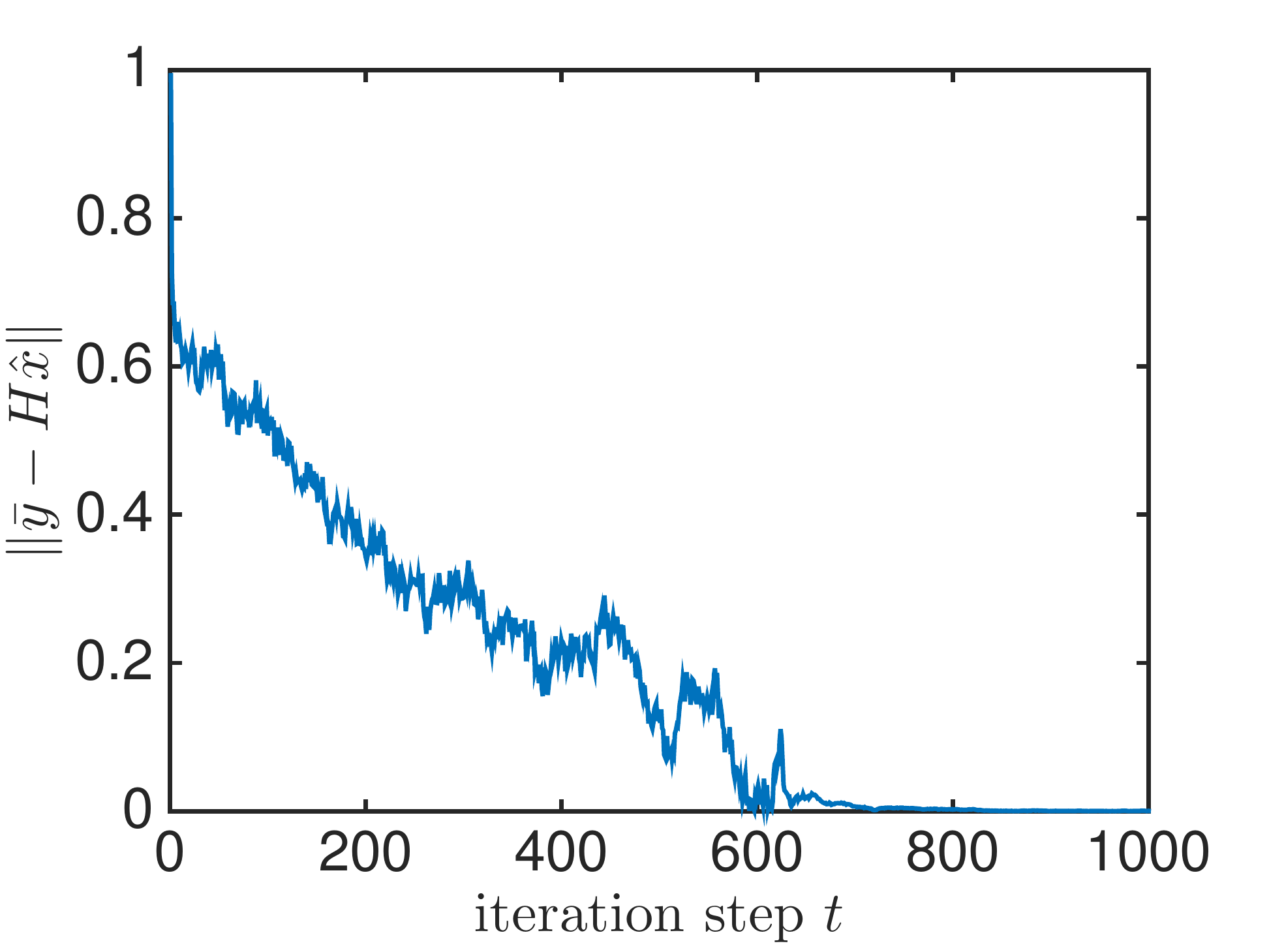}
\caption{The evolution of the magnitude of innovation $\left\| \bar{\by} -H\hat{\bx} \right\|$ with resampling implemented.}
\label{fig:fig_misfit_with_resampling}
\end{figure}

\begin{figure}[h]
\centering
\includegraphics[width=0.45\columnwidth]{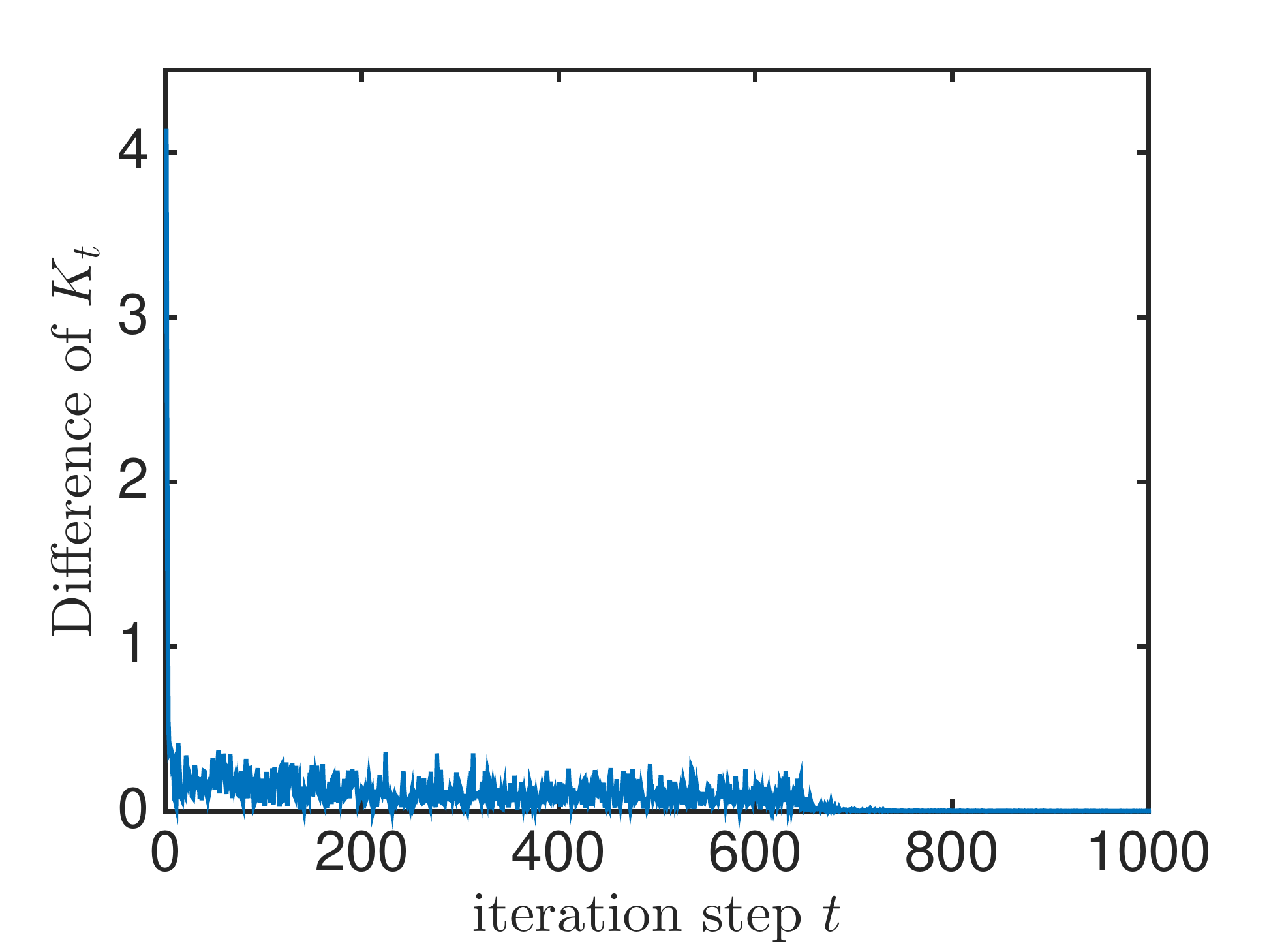}
\caption{The difference between the Kalman gains before and after resampling: $\left\| K_{t, before} - K_{t, after}  \right\|$.}
\label{fig:fig_diff_K_t}
\end{figure}

\begin{figure}[h]
\centering
\includegraphics[width=0.45\columnwidth]{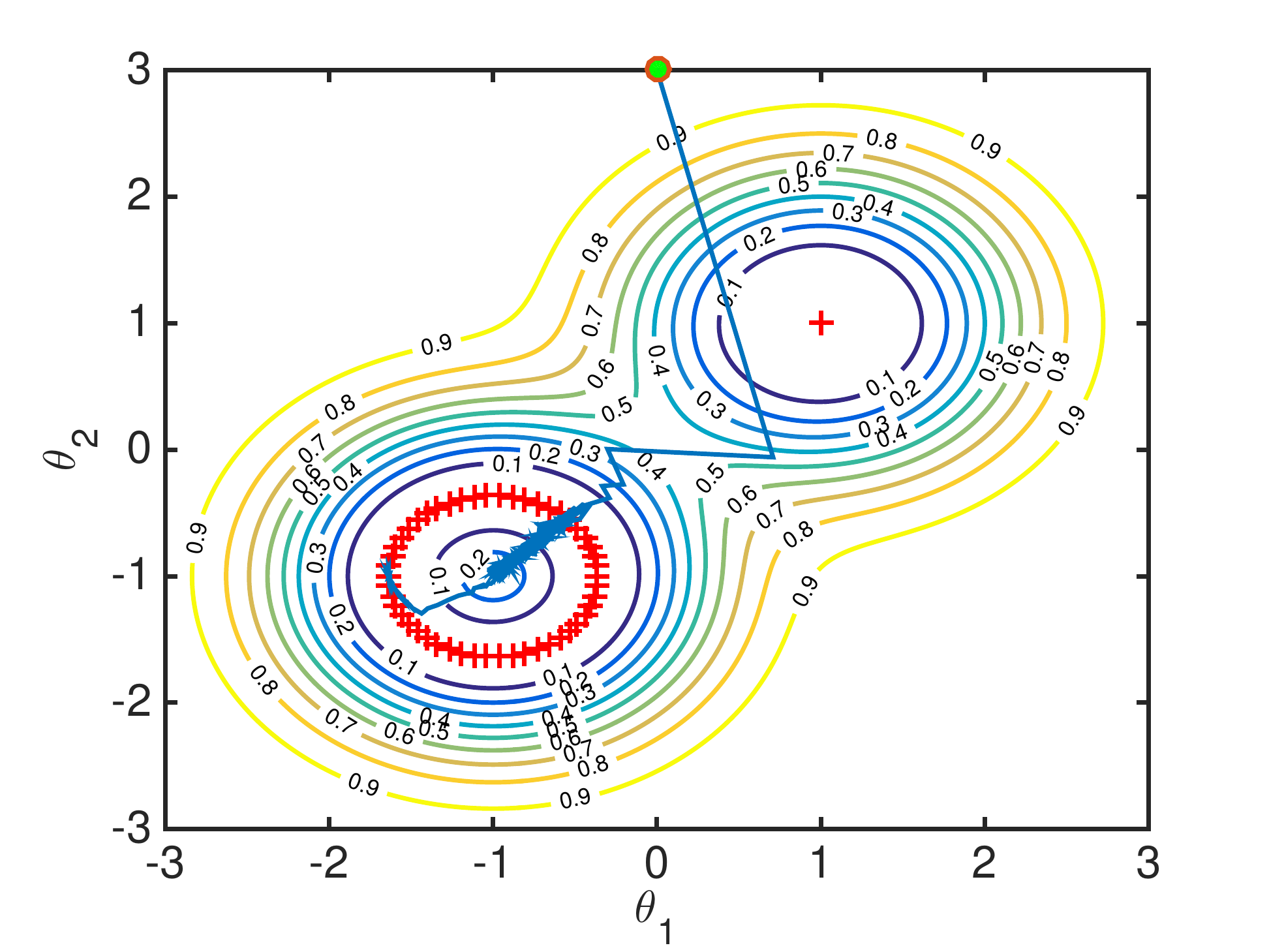}
\caption{The solution path of the ensemble mean of $\bth$ with resampling implemented. The red ``+" represent the $\bth$ values that minimize $\left\|\bar{\by} - H\bfor(\bth) \right\|^2$. The green circle denotes the starting point.}
\label{fig:fig_theta_convergence_with_resampling}
\end{figure}

\begin{figure}[h]
\centering
\includegraphics[width=0.45\columnwidth]{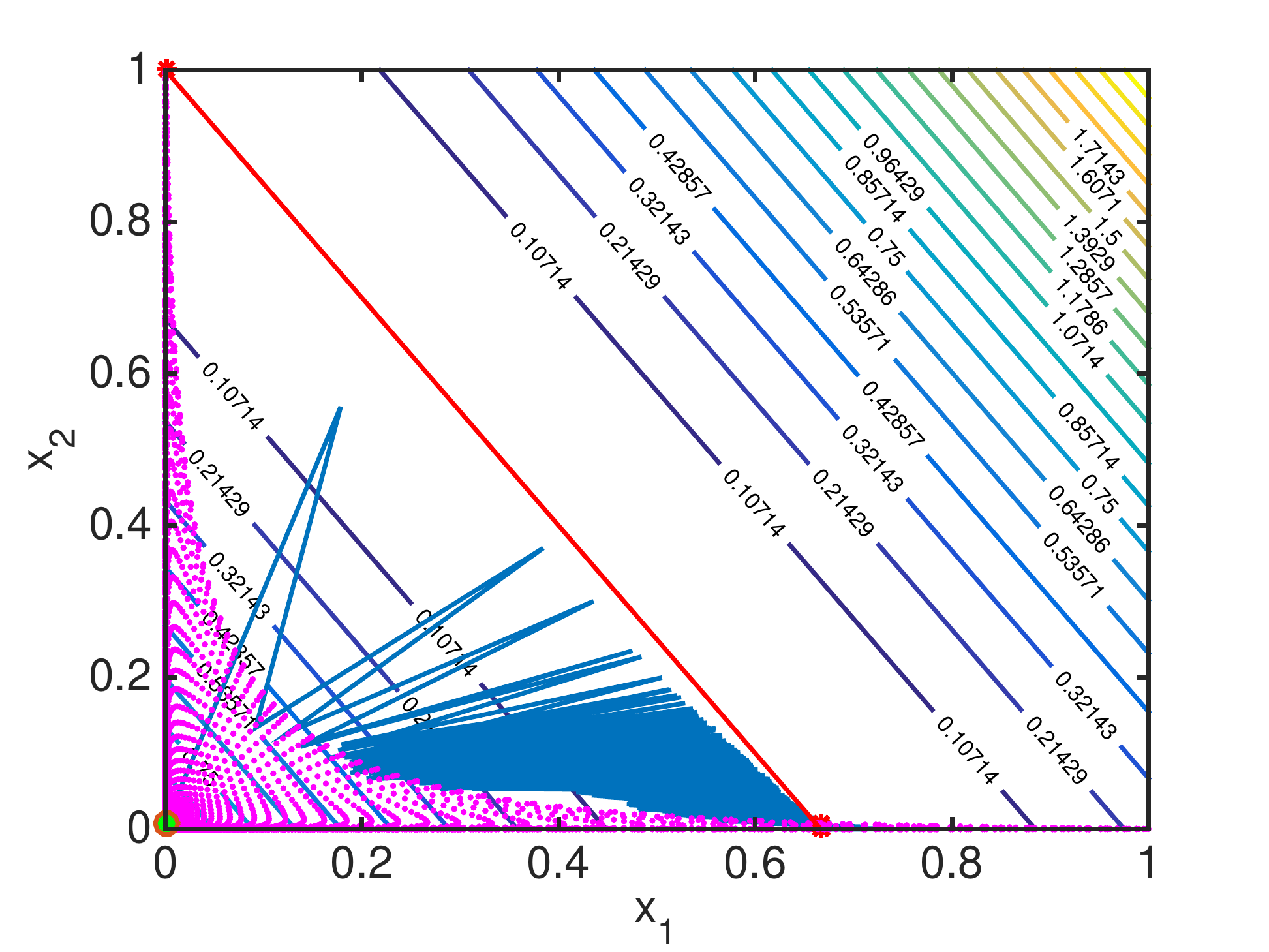}
\caption{The solution path of the ensemble mean of $x$ with resampling implemented. The red line in the middle represents the level sets of $\bar{\by} - Hx = 0$. The dotted region in the lower left corner represents the image of the forward mapping $\bfor(\cdot)$. The green circle denotes the starting point.}
\label{fig:fig_x_convergence_with_resampling}
\end{figure}

\subsection{Influence of Higher Order Moments on Convergence}\label{sec:kurtosis}
In the example application above, a Gaussian distribution was used during the resampling process, with the first and second order moments kept unchanged from the posterior distribution of the previous iteration step. However, other distributions with different higher order moments can also be used to resample the ensemble. It would be unusual to consider skewness (the third order moment) of the resampling distribution without prior knowledge, thus we mainly focus here on the influence of kurtosis (the fourth order moment) on the convergence of the proposed IREnKF. Namely, three different resampling distributions will be considered (see Fig.~\ref{fig:fig_kurtosis_different_distribution}): 
\begin{itemize}
\item Uniform distribution: kurtosis = 1.8;
\item Gaussian distribution: kurtosis = 3;
\item Laplace distribution: kurtosis = 6.
\end{itemize}
These distributions have zero skewness, and the mean and variance are set to be the same as the ensemble before resampling. The numerical simulation results for the above resampling distributions are shown in Fig.~\ref{fig:kurtosis_resampling_theta_convergence}, which yield two main conclusions:
\begin{itemize}
\item Larger kurtosis makes the convergence more oscillatory. In Fig.\ref{fig:kurtosis_resampling_theta_convergence}, the main trends of the convergence for uniform and gaussian resampling (smaller kurtosis) are almost monotonic while the convergence for Laplace resampling (larger kurtosis) is more oscillatory. This is because the distribution with larger kurtosis is more likely to generate outliers, which will be reflected by increased fluctuations in the ensemble mean parameter estimation. This is not necessarily a disadvantage, however, as this naturally adds mutation that may help the solution to avoid local minima and converge to the global minimum.
\item Smaller kurtosis accelerates the convergence rate. It can also be observed in Fig.~\ref{fig:kurtosis_resampling_theta_convergence} that Laplace resampling (with the largest kurtosis) yields the largest iteration number for convergence (about 1200 steps). Even comparing uniform resampling and Gaussian resampling, which yield almost the same convergence (about 600 steps), we can find that Gaussian resampling (with larger kurtosis) takes more iterations to converge to the neighborhood of local minimums around $(-1, -1)$ than uniform resampling. This is because resampling distribution with smaller kurtosis is less likely to generate outliers and most of the new ensembles will be near the ensemble mean, which is in general a better estimator of the parameter value than an individual ensemble member. 
\end{itemize} 

\begin{figure}[h]
\centering
    \includegraphics[width=0.45\columnwidth]{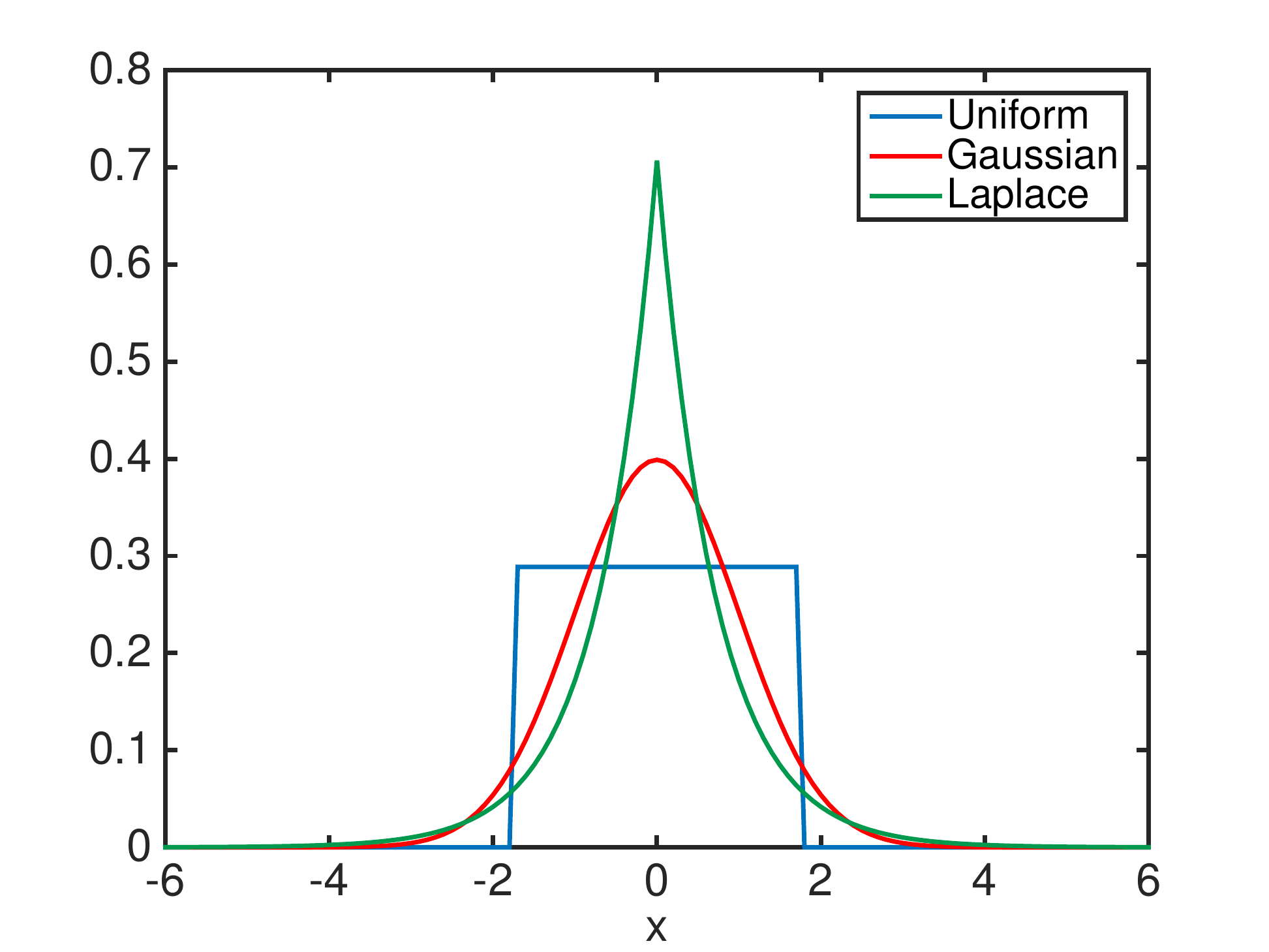}
    \caption{The probability density functions of the three distributions used for resampling. Uniform distribution: kurtosis = 1.8\;. Gaussian distribution: kurtosis = 3\;. Laplace distribution: kurtosis = 6\;.}
    \label{fig:fig_kurtosis_different_distribution}
\end{figure}

\begin{figure}[h]
  \centering
  \begin{tabular}{ccc}
    \includegraphics[width=0.33\columnwidth]{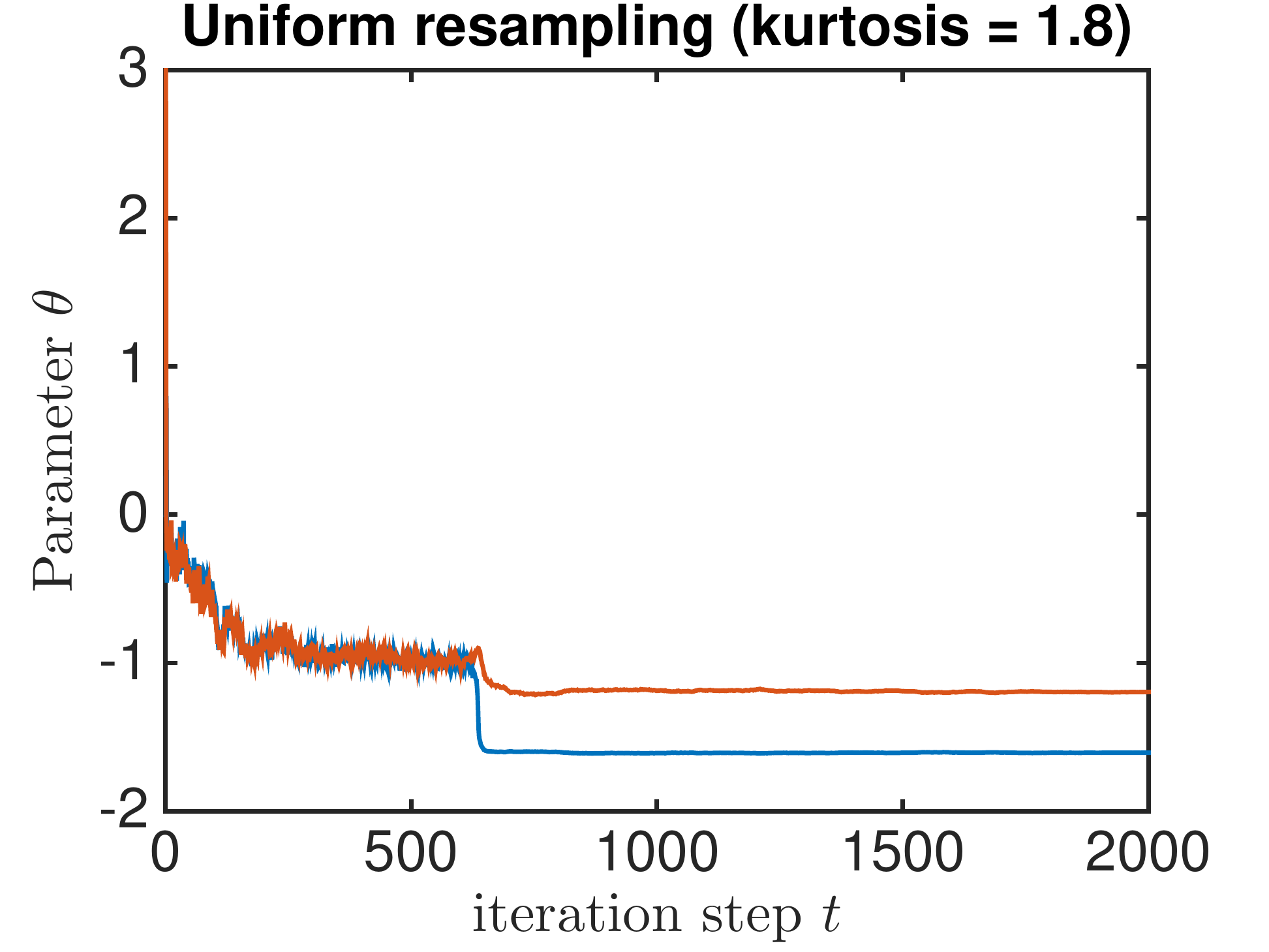}&
    
    \includegraphics[width=0.33\columnwidth]{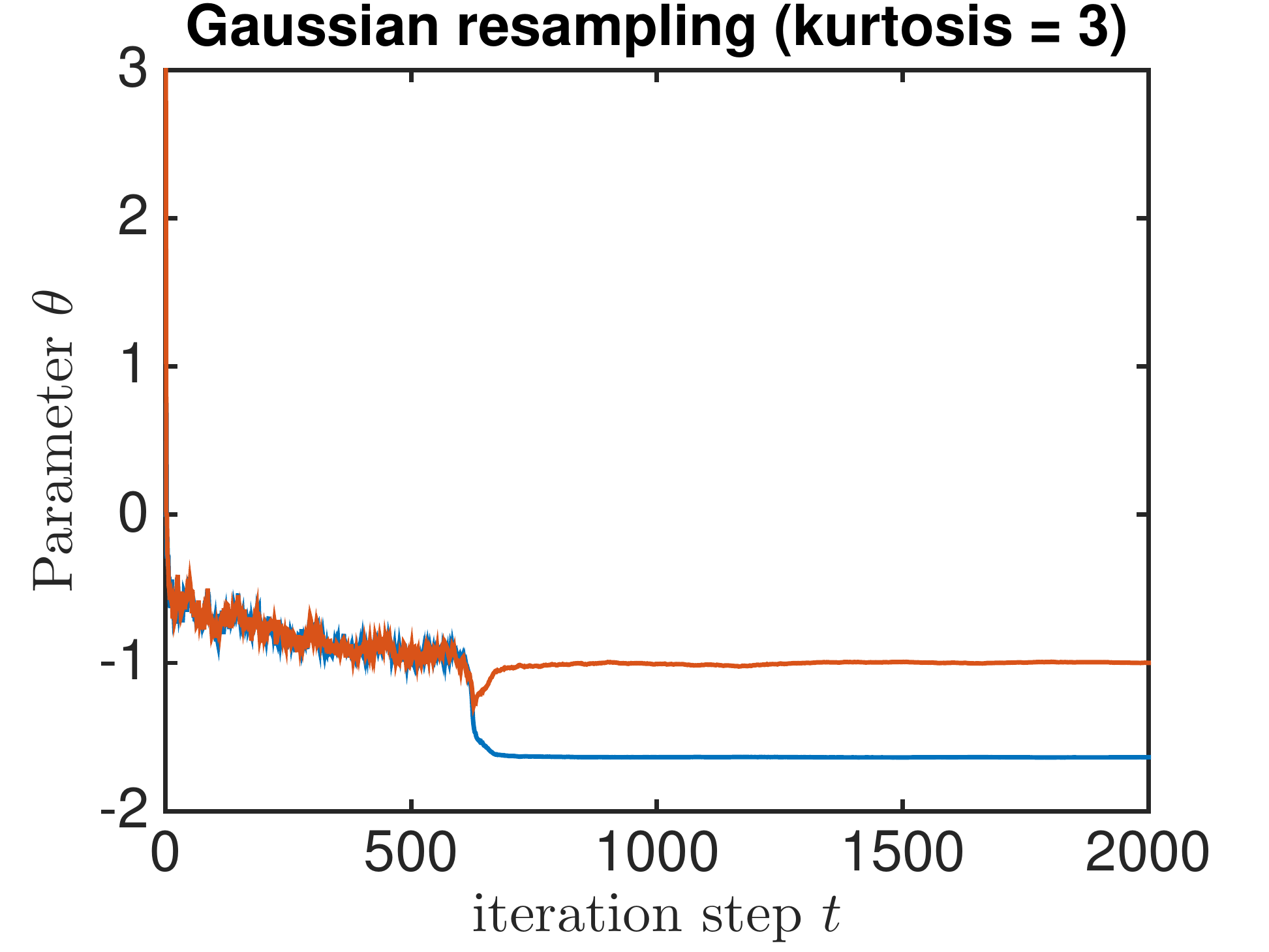}&
    
    \includegraphics[width=0.33\columnwidth]{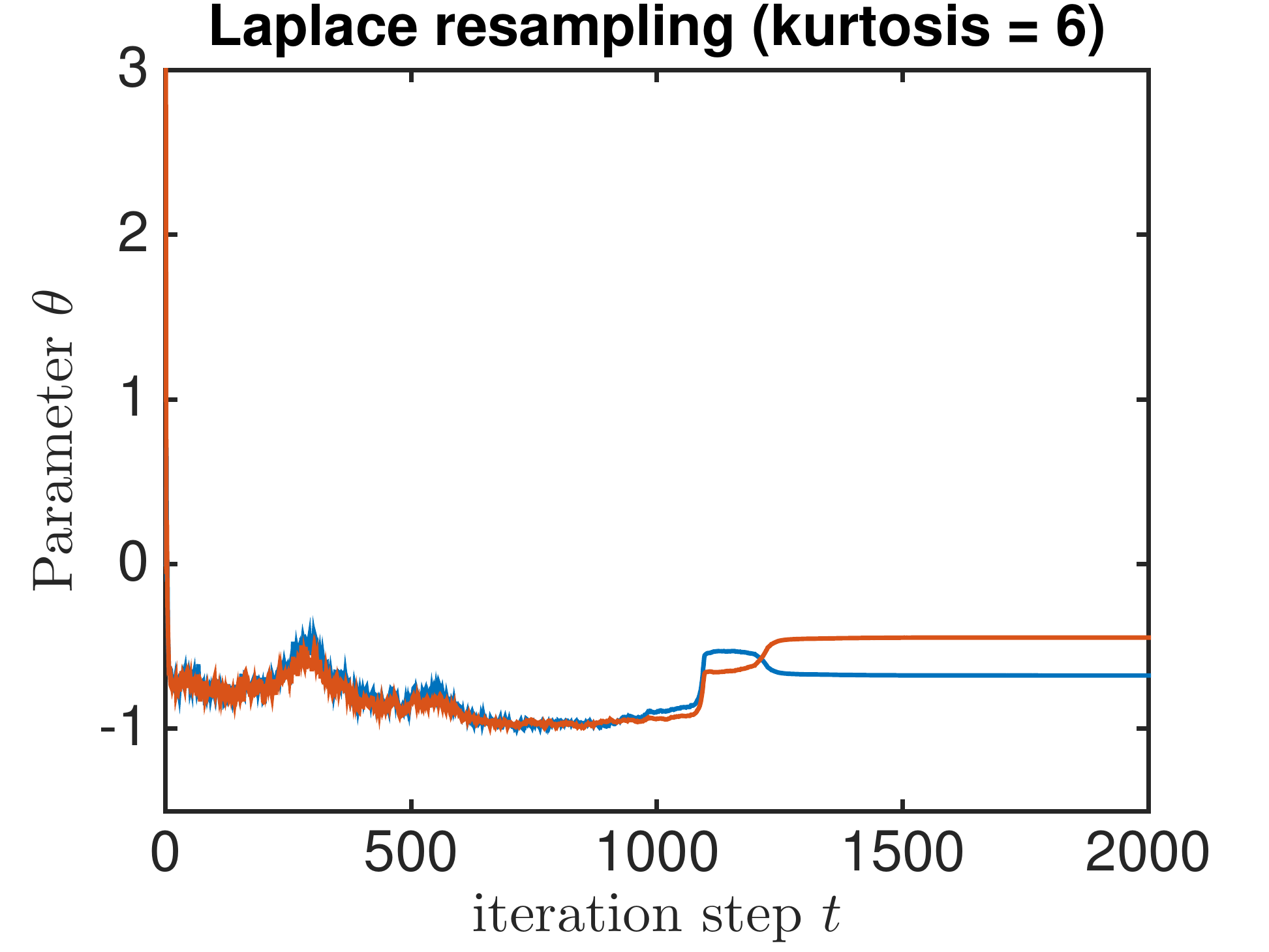}\\
  \end{tabular}
  \caption{The convergence of parameter $\bth$ estimation with different resampling distributions implemented. (Left) Uniform distribution. (Middle) Gaussian distribution. (Right) Laplace distribution.}
\label{fig:kurtosis_resampling_theta_convergence}
\end{figure}

\section{Conclusion}
Herein a comprehensive study of the convergence of iterative ensemble Kalman filter is conducted and a potential method to improve its convergence was proposed. We first showed the necessity of iterations for ensemble Kalman method applied to nonlinear inverse problems by reformulating it as a constrained optimization minimizing the innovation $\left\|\by - H \bx \right\|$ while satisfying the forward model equation $\bx = \mathbf{f}(\bth)$. Then the shrinking effect of the standard Kalman iterations on ensemble covariances was studied and, along with the nonlinearity of the forward model, was shown to be the reason the Kalman gain approaches zero before innovation was minimized, i.e. leading to the early stopping of the IEnKF. To resolve this issue, we added an additional step at each iteration to resample the parameter posterior ensemble with the mean and covariance kept unchanged before assigning as prior to the next iteration. The idea is to prevent early stopping by perturbing the shrinking of the ensemble covariances while still keeping the correct Kalman update direction. Consideration of high moments of the resampling distribution demonstrated that higher kurtosis adds more mutations each iteration, which may be advantageous to avoid local minima, whereas smaller kurtosis concentrates ensemble members around the mean and thus speeds convergence when the ensemble mean is a good estimator of the solution. Numerical simulations were presented to reproduce the early stopping phenomenon and demonstrate the effectiveness of the proposed resampling method.


\end{document}